\documentclass[12pt]{amsart}
\usepackage{amssymb,comment,url}  
\usepackage[all]{xy}

\newcommand{\Aff}{{\mathbb A}}
\newcommand{\Aut}{{\operatorname{Aut}}}
\newcommand{\Br}{{\operatorname{Br}}}
\newcommand{\C}{{\mathbb C}}
\newcommand{\cc}{{\mathbf c}}
\newcommand{\D}{{\mathbf D}}
\newcommand{\Diag}{{\operatorname{Diag}}}
\newcommand{\Gal}{{\operatorname{Gal}}}
\newcommand{\GL}{{\operatorname{GL}}}
\newcommand{\Gm}{{\mathbb G}_m}
\newcommand{\Hbar}{{\overline{H}}}
\newcommand{\Hom}{{\operatorname{Hom}}}
\newcommand{\isom}{{\, \cong \,}}
\newcommand{\Jac}{{\operatorname{Jac}}}
\newcommand{\K}{{\mathbb K}}
\newcommand{\Kbar}{{\overline{K}}}
\newcommand{\LL}{{\mathcal L}}
\newcommand{\la}{{\lambda}}

\newcommand{\Ob}{{\operatorname{Ob}}}
\newcommand{\PP}{{\mathbb P}}
\newcommand{\PGL}{{\operatorname{PGL}}}
\newcommand{\Q}{{\mathbb Q}}
\newcommand{\ra}{{\, \longrightarrow \,}}
\newcommand{\rank}{{\operatorname{rank}}}
\newcommand{\phibar}{{\overline{\varphi}}}
\newcommand{\ratto}{{\,- \to\,}}
\newcommand{\rhobar}{{\overline{\rho}}}
\newcommand{\Sec}{{\operatorname{Sec}}}
\newcommand{\Sha}{\mbox{\wncyr Sh}}
\newcommand{\SL}{{\operatorname{SL}}}
\newcommand{\Tr}{\operatorname{Tr}}
\newcommand{\Z}{{\mathbb Z}}

\newfont{\wncyr}{wncyr10 at 12pt}
\newfont{\wncyrten}{wncyr10 at 10pt}

\def\MatP{{\begin{pmatrix}
    1   &    0   &    0    & \cdots &      0     \\
    0   & \zeta_n  &    0    & \cdots &      0     \\
    0   &    0   & \zeta_n^2 & \cdots &      0     \\ 
 \vdots & \vdots & \vdots  &        &   \vdots   \\
    0   &   0    &    0    & \cdots & \zeta_n^{n-1}  
\end{pmatrix}}}

\def\MatQ{{\begin{pmatrix}
    0   &    0   & \cdots &   0    &      1     \\ 
    1   &    0   & \cdots &   0    &      0     \\
    0   &    1   & \cdots &   0    &      0     \\ 
 \vdots & \vdots &        & \vdots &   \vdots   \\
    0   &   0    & \cdots &   1    &      0       
\end{pmatrix}}}

\newenvironment{Proof}{\par\noindent{\sc Proof:}}%
                      {\hspace*{\fill}\nobreak$\Box$\par\medskip}
                      {\hspace*{\fill}\nobreak$\Box$\par\medskip}

\newtheorem{Proposition}{Proposition}[section]
\newtheorem{Theorem}[Proposition]{Theorem}
\newtheorem{Lemma}[Proposition]{Lemma}
\newtheorem{Corollary}[Proposition]{Corollary}

\newtheorem{Conjecture}[Proposition]{Conjecture}

\theoremstyle{definition}

\newtheorem{Definition}[Proposition]{Definition}
\newtheorem{Remark}[Proposition]{Remark}
\newtheorem{Remarks}[Proposition]{Remarks}
\newtheorem{Example}[Proposition]{Example}

\renewcommand{\theenumi}{\roman{enumi}}

\begin{document}

\date{14th October 2011}
\title[Invariant theory for the elliptic normal quintic]%
{Invariant theory for the~elliptic~normal~quintic,\\ 
I. Twists of X(5)}
\author{Tom Fisher}
\address{University of Cambridge,
         DPMMS, Centre for Mathematical Sciences,
         Wilberforce Road, Cambridge CB3 0WB, UK}
\email{T.A.Fisher@dpmms.cam.ac.uk}

\begin{abstract}
A genus one curve of degree $5$ is defined by the $4 \times 4$ Pfaffians
of a $5 \times 5$ alternating matrix of linear forms on $\PP^4$. We
describe a general method for investigating the invariant theory of
such models. We use it to explain how we found our algorithm for 
computing the invariants~\cite{g1inv} and to extend our method 
in~\cite{g1hess} for computing equations for visible elements of 
order $5$ in the Tate-Shafarevich group of an elliptic curve. As a
special case of the latter we find a formula for the family of elliptic
curves $5$-congruent to a given elliptic curve in the case the
$5$-congruence does {\em not} respect the Weil pairing. We also give
an algorithm for doubling elements in the $5$-Selmer group of an 
elliptic curve, and make a conjecture about the matrices representing
the invariant differential on a genus one normal curve of arbitrary degree.
\end{abstract}

\maketitle

\section{Introduction}
\label{sec:intro}

A genus one normal curve $C \subset \PP^{n-1}$ of degree $n \ge 3$ 
is a genus one curve embedded by a complete linear system of 
degree $n$. If $n \ge 4$ then the homogeneous ideal of $C$ is generated
by a vector space of quadrics of dimension $n(n-3)/2$. 
\begin{Definition}
A {\em genus one model} (of degree $5$) is a $5 \times 5$ alternating 
matrix of linear forms on $\PP^4$. We write $C_\phi \subset \PP^4$
for the subvariety defined by the $4 \times 4$ Pfaffians of $\phi$,
and say that $\phi$ is {\em non-singular} if $C_\phi$ is a smooth curve
of genus one.
\end{Definition}
It is a classical fact that (over $\C$) every genus one normal 
curve of degree~5
is of the form $C_\phi$ for some $\phi$. Importantly for us, the proof 
using the Buchsbaum-Eisenbud structure theorem~\cite{BE1},\cite{BE2} 
shows that this is still true over an arbitrary ground field.

There is a natural action of 
$\GL_5 \times \GL_5$ on the space of genus one models. The first
factor acts as $M : \phi \mapsto M \phi M^T$ and the second factor
by changing co-ordinates on $\PP^4$. To describe this
situation we adopt the following notation. Let $V$ and $W$ be 
$5$-dimensional vector spaces with bases $v_0, \ldots, v_4$ and
$w_0, \ldots, w_4$. 
The dual bases for $V^*$ and $W^*$ will be
denoted $v^*_0, \ldots, v^*_4$ and $w^*_0, \ldots, w^*_4$.
We identify the space of genus one models with
$\wedge^2 V \otimes W$ via
\[ \phi = (\phi_{ij})  \longleftrightarrow \textstyle\sum_{i<j} (v_i \wedge v_j) 
\otimes \phi_{ij}(w_0, \ldots, w_4). \]
With this identification the action of $\GL_5 \times \GL_5$ becomes 
the natural action of $\GL(V) \times \GL(W)$ on $\wedge^2 V \otimes W$.
By squaring and then identifying $\wedge^4 V \isom V^*$ there
is a natural map
\begin{equation}
\label{def:P2}
 P_2 : \wedge^2 V \otimes W \to V^* \otimes S^2 W = \Hom(V,S^2W). 
\end{equation}
Explicitly $P_2(\phi) = (v_i \mapsto p_i(w_0,\ldots, w_4))$ 
where $p_0, \ldots, p_4$ are the $4 \times 4$ Pfaffians of $\phi$.
Thus $V$ may be thought of as the space of quadrics defining $C_{\phi}$
and $W$ as the space of linear forms on $\PP^4$. Despite this clear
geometric distinction, we show in this paper 
that certain covariants that mix up the roles of $V$ and $W$ have
interesting arithmetic applications.

We work over a perfect field $K$ with characteristic not equal to 
$2$, $3$ or $5$. The co-ordinate ring $K[\wedge^2 V \otimes W]$ 
is a polynomial ring in $50$ variables.

\begin{Theorem}
\label{thm:inv} 
The ring of invariants for 
$\SL(V) \times \SL(W)$ acting on $K[\wedge^2 V \otimes W]$ 
is generated by invariants $c_4$ and $c_6$ 
of degrees $20$ and $30$. If we scale them as specified in \cite{g1inv}
and put $\Delta = (c_4^3 - c_6^2)/1728$ then
\begin{enumerate}
\item A genus one model $\phi$ is non-singular if and only if 
$\Delta(\phi) \not=0$.
\item If $\phi$ is non-singular then $C_\phi$ has Jacobian elliptic
curve 
\[ y^2 = x^3- 27 c_4(\phi) x -54 c_6(\phi). \]
\end{enumerate}
\end{Theorem}
\begin{Proof}
See \cite[Theorem 4.4]{g1inv}.
\end{Proof}

The invariants $c_4$ and $c_6$ are too large to write down 
as explicit polynomials. Nonetheless we gave an algorithm 
for evaluating them in \cite[Section 8]{g1inv}. By Theorem~\ref{thm:inv}
this gives an algorithm for computing the Jacobian. In \cite{g1hess}
we studied a covariant we call the Hessian. Explicitly it is a 
$50$-tuple of homogeneous polynomials of degree $11$ defining a map
$H : \wedge^2 V \otimes W \to \wedge^2 V \otimes W$. Again rather
than write down these polynomials we gave an algorithm
for evaluating them. The Hessian allows us to compute certain 
twists of the universal family of elliptic curves parametrised by $X(5)$. 
We used it to find equations for visible elements of order $5$ in
the Tate-Shafarevich group of an elliptic curve, and to recover the
formulae of Rubin and Silverberg \cite{RubSil1} for families of $5$-congruent
elliptic curves. However in both these applications we were restricted
to $5$-congruences that respect the Weil pairing. In this paper we
remove this restriction. 

In Sections~\ref{sec:covs},~\ref{sec:heis},~\ref{sec:discov} 
we explain a general
method for investigating the covariants associated to a genus one
model. In particular we explain how we found 
the algorithm for computing the invariants in \cite[Section 8]{g1inv}.
One key result on the existence of covariants is left to 
a sequel to this paper \cite{paperII}. 
Our account is still however self-contained, since
in Section~\ref{sec:constr} we give explicit constructions of each of the
covariants used in the second half of this paper. In 
Section~\ref{sec:congr} we use the covariants to write
down families of $5$-congruent elliptic curves. In
Section~\ref{sec:double} we give a formula for doubling in the $5$-Selmer
group of an elliptic curve and extend our method in \cite{g1hess} for
computing visible elements of the Tate-Shafarevich group.
In particular we check local solubility for each of the visible 
elements of order~$5$ in the Weil-Ch\^atelet group that were considered
in \cite{CM}. 
In Section~\ref{sec:invdiff} we study a covariant that describes the 
invariant differential. This is needed not only for some of the constructions 
in Section~\ref{sec:constr}, but also leads us to make a conjecture 
about the matrices representing the invariant differential on 
a genus one normal curve of arbitrary degree.

\section{Covariants}
\label{sec:covs}

We recall that a {\em rational representation} of a linear algebraic
group $G$ is a morphism of group varieties $\rho_Y : G \to \GL(Y)$.

\begin{Definition}
\label{def:cov}
Let $Y$ be a rational representation of $\GL(V) \times \GL(W)$.
A {\em covariant} (for $Y$) is a polynomial map 
$F: \wedge^2 V \otimes W \to Y$ such that $F \circ g = \rho_Y(g) F$
for all $g \in \SL(V) \times \SL(W)$.
\end{Definition}

The covariants in the case $Y=K$ is the trivial representation are
the invariants as described in Theorem~\ref{thm:inv}.
For general $Y$ the covariants form a module over the ring
of invariants. 
In all our examples $Y$ will be {\em homogeneous} by which 
we mean there exist integers $r$ and $s$ such that
\[ \rho_Y(\lambda I_V, \mu I_W)= \lambda^r \mu^s I_Y \]
for all $\lambda, \mu \in K^\times$. A polynomial map 
$F: \wedge^2 V \otimes W \to Y$ is {\em homogeneous} of degree $d$
if $F(\lambda \phi)=\lambda^d F(\phi)$ for all $\lambda \in K$,
equivalently $F$ is represented by a tuple of homogeneous polynomials 
of degree $d$.
 
\begin{Lemma}
\label{wtlemma}
Let $F: \wedge^2 V \otimes W \to Y$ be a covariant and suppose 
that both $Y$ and $F$ are homogeneous. Then there exist integers $p$ and $q$ 
called the {\em weights} of $F$ such that
\begin{equation*}
 F \circ g = (\det g_V)^p (\det g_W)^q \rho_Y(g) \circ F 
\end{equation*}
for all $g = (g_V,g_W) \in  \GL(V) \times \GL(W)$.
Moreover if $Y$ has degree $(r,s)$ then 
\begin{equation}
\label{wtformula}
\begin{array}{rcl}
2 \deg F & = & 5p+r \\ 
\deg F & = & 5q+s.
\end{array} 
\end{equation}
\end{Lemma}
\begin{Proof}
The only 1-dimensional rational representations 
of $\GL_n$ are integer powers of the determinant. This proves
the first statement. The second statement follows
from the special case where $g_V$ and $g_W$ are scalar matrices. 
\end{Proof}

The first example of a covariant is the identity map
\[U:\wedge^2 V \otimes W \to \wedge^2 V \otimes W.\] 
It has degree 1 and weights $(p,q)=(0,0)$. The Pfaffian map
$P_2$ defined in~(\ref{def:P2}) is a covariant of  
degree 2 with weights $(p,q)=(1,0)$. Subject to picking a 
basis for $V$, $\phi \in \wedge^2 V \otimes W$ is a
$5 \times 5$ alternating matrix of linear forms and 
$P_2(\phi)$ is its vector of $4 \times 4$ Pfaffians.
The determinant of the Jacobian matrix of these 
$5$ quadrics defines a covariant
\begin{equation}
\label{def:S10}
S_{10}:\wedge^2 V \otimes W \to S^5 W. 
\end{equation}
It has degree 10 and weights $(p,q)=(4,1)$. It is shown in 
\cite[VIII.2.5]{Hu} that if $\phi \in \wedge^2 V \otimes W$ is 
non-singular then $S_{10}(\phi)$ is an equation for the secant variety 
of $C_\phi \subset \PP^4$.

Our initial motivation for studying the covariants was that
by constructing a large enough supply of covariants we might
eventually arrive at an algorithm for computing the invariants, and
so by Theorem~\ref{thm:inv} an algorithm for computing the  
Jacobian. This programme was successful, leading to the algorithm 
in \cite{g1inv}. We have subsequently found 
that some of the covariants have interesting arithmetic 
applications in their own right.

In the next two sections we explain our methods for studying the covariants.
The key idea is that although the covariants are routinely too
large to write down, their restrictions to the Hesse family,
i.e. the universal family of elliptic curves over $X(5)$,
are much easier to write down and are (nearly) characterised by their 
invariance properties under an appropriate action of 
$\SL_2(\Z/5\Z)$. Thus our work resolves, albeit in one particular
case, what is described in \cite[Chapter V,\S22]{AR}
as the ``mysterious role of invariant theory''.

\section{The extended Heisenberg group}
\label{sec:heis}

We take $n \ge 5$ an odd integer. In this section we work over 
an algebraically closed field $K$ of characteristic not dividing $n$, 
and let $\zeta_n \in K$ be a primitive $n$th root of unity. We write
$E[n]$ for the $n$-torsion subgroup of an elliptic curve $E$ and
$e_n : E[n] \times E[n] \to \mu_n$ for the Weil pairing.

\begin{Definition}
(i) The modular curve $Y(n)= X(n) \setminus \{ \text{cusps} \}$ 
parametrises triples $(E, P_1,P_2)$ where $E$ is an elliptic curve 
and $P_1,P_2$ are a basis for $E[n]$ with $e_n(P_1,P_2) = \zeta_n$. \\
(ii) Let $Z(n) \subset \PP^{n-1}$ be the subvariety defined by
$a_0 = 0$, $a_{n-i} = -a_i$ and $\rank (a_{i-j} a_{i+j}) \le 2$
where $(a_0 : \ldots :a_{n-1})$ are co-ordinates on $\PP^{n-1}$ and
the subscripts are read mod $n$.
\end{Definition}

There is an action of $\SL_2(\Z/n\Z)$ on $Y(n)$ given by
\[ \left( \begin{matrix} a & b \\ c & d \end{matrix} \right) :
(E,P_1,P_2) \mapsto (E,d P_1 - c P_2,-b P_1 + a P_2). \]
Let
$S = \left( \begin{smallmatrix} 0 & 1 \\ - 1 & 0 \end{smallmatrix} \right)$ 
and 
$T = \left( \begin{smallmatrix} 1 & 1 \\ 0 & 1 \end{smallmatrix} \right)$ 
be the usual generators for $\SL_2(\Z)$. By abuse of notation we also
write $S$ and $T$ for their images in $\SL_2(\Z/n\Z)$.

\begin{Theorem}
\label{thm:rho}
There is an embedding $X(n) \subset \PP^{n-1}$ such that 
\begin{enumerate}
\item $X(n) \subset Z(n)$ with equality if $n$ is prime.
\item The action of $\SL_2(\Z/n\Z)$ on $X(n)$ is given by
$\rhobar : \SL_2(\Z/n\Z) \to \PGL_n(K)$ where
\[ \qquad \rhobar(S) \propto (\zeta_n^{ij})_{i,j=0,\ldots,n-1} 
\qquad \rhobar(T) \propto \Diag(\zeta_n^{i^2/2})_{i=0,\ldots,n-1} \]
\end{enumerate}
Moreover $\rhobar$ lifts uniquely to a representation 
$\rho : \SL_2(\Z/n\Z) \to \SL_n(K)$. 
\end{Theorem}

\begin{Proof}
The condition for equality in (i) is due to
V\'elu \cite{V}. The remaining statements are proved in 
\cite[Section 2]{highercongr}. Proofs that $\rhobar$ lifts
in the case $n=5$ may also be found in \cite{HM}, \cite{SBT}.
\end{Proof}

The Heisenberg group of level $n$ is 
\[ H_n = \langle \sigma, \tau | \sigma^n = \tau^n = [ \sigma, [\sigma,\tau]]
= [\tau, [\sigma,\tau]] = 1 \rangle. \]
It is a non-abelian group of order $n^3$.
The centre is a cyclic group of order $n$ generated by
$\zeta=[\sigma,\tau] = \sigma \tau \sigma^{-1} \tau^{-1}$.
We write $\Hbar_n$ for the quotient of $H_n$ by its centre 
and identify $\Hbar_n \isom (\Z/n\Z)^2$ via 
$\overline{\sigma} \mapsto (1,0)$ and $\overline{\tau} \mapsto (0,1)$.
Since each automorphism of $H_n$ induces an automorphism of $\Hbar_n$
there is a natural group homomorphism
$\beta: \Aut(H_n) \to \GL_2(\Z/n\Z)$. The kernel of $\beta$ is $\Hbar_n$ acting 
as the group of inner automorphisms. We may thus 
identify $\Aut(H_n)$ as a group of affine transformations.
Let $\iota \in \Aut(H_n)$ be the involution given by
$\iota(\sigma) = \sigma^{-1}$ and $\iota(\tau) = \tau^{-1}$. 
(Any involution $\iota$ with $\beta(\iota) = -I$ would do, but we have
picked one for definiteness.)
Since $n$ is odd there is a unique section $s_\beta$ for $\beta$
with $s_\beta(-I)=\iota$. (This means that
$s_\beta : \GL_2(\Z/n\Z) \to \Aut(H_n)$ is a group homomorphism 
with $\beta \circ s_\beta = {\rm id}$.) 
Indeed the image of $s_\beta$ is the centraliser of $\iota$
in $\Aut(H_n)$. 

\begin{Definition}
\label{defexheisen}
The extended Heisenberg group is the semi-direct product
\[ H_n^+ = H_n \ltimes \SL_2(\Z/n\Z), \]
with group law $ (h,\gamma) (h',\gamma')= (h \, s_\beta(\gamma)h', 
\gamma \gamma')$.
\end{Definition}

\begin{Remark}
(i) The map $s_\beta$ 
is explicitly given by \[ s_\beta ((
 \begin{smallmatrix} a & b \\ c & d \end{smallmatrix} )) 
: \sigma \mapsto \zeta^{-ac/2} \sigma^a \tau^c \, ;
\, \, \tau \mapsto \zeta^{-bd/2} \sigma^b \tau^d.\]
(ii) The group $H_5^+$ was used by Horrocks and Mumford \cite{HM}
in their construction of an indecomposable rank 2 vector bundle
on $\PP^4$. In fact the order of $H_5^+$ appears in the 
title of their paper. 
\end{Remark}

We now identify $H_n$ as a subgroup of $\SL_n(K)$ via 
the Schr\"odinger representation $\theta : H_n \to \SL_n(K)$ where
\begin{equation}
\label{heismats}
 \theta(\sigma) = \MatP, \quad \theta(\tau) = \MatQ. 
\end{equation}
Since $\theta(\zeta) = \zeta_n I_n$ this identifies the centre
of $H_n$ with $\mu_n$. Let $N_n$ be the
normaliser of $H_n$ in $\SL_n(K)$. It may be checked that the
automorphisms of $\Hbar_n$ induced by conjugation by
elements of $N_n$ are precisely those that preserve the 
commutator pairing $\Hbar_n \times \Hbar_n \to \mu_n$.
This proves the surjectivity of the map $\alpha$
in the following commutative diagram with exact rows and columns.
\begin{equation*}
\xymatrix{& 0 \ar[d] & 0 \ar[d] \\ 
& \mu_n \ar@{=}[r] \ar[d] & \mu_n \ar[d] & 0 \ar[d] \\
0 \ar[r] & H_n     \ar[r] \ar[d] & N_n \ar[r]^-\alpha \ar[d]
 &  \SL_2(\Z/n\Z) \ar[d] \ar[r] & 0 \\
0 \ar[r] & \Hbar_n  \ar[d] \ar[r] & \Aut(H_n) \ar[d] \ar[r]^-\beta  
&  \GL_2(\Z/n\Z) \ar[d]^{\det} \ar[r] & 0 \\ & 0 & (\Z/n\Z)^\times 
\ar@{=}[r] \ar[d] & (\Z/n\Z)^\times \ar[d] \\ & & 0 & 0    }
\end{equation*}

The restriction of $s_\beta$ to $\SL_2(\Z/n\Z)$ defines
a projective representation 
\begin{equation*}
 \rhobar : \SL_2(\Z/n\Z) \to \PGL_n(K). 
\end{equation*}
Comparison with the proof of Theorem~\ref{thm:rho} shows that 
this is the same as the projective representation considered there.

\begin{Lemma}
\label{lem:taut}
The following objects are in natural 1-1 correspondence. 
\begin{enumerate}
\item Sections $s_\alpha$ for $\alpha$ compatible with $s_\beta$. 
\item Lifts of $\rhobar:\SL_2(\Z/n\Z) \to \PGL_n(K)$ to 
$\rho:\SL_2(\Z/n\Z) \to \SL_n(K)$. 
\item Extensions of $\theta:H_n \to \SL_n(K)$ to $\theta^+:H_n^+ \to \SL_n(K)$.
\end{enumerate}
\end{Lemma}

\begin{proof}
The projective representation $\rhobar$ is defined by the
requirement
\begin{equation*}
\rhobar(\gamma) \, h \, \rhobar(\gamma)^{-1} = s_\beta (\gamma) h  
\quad \text{ for all }
h \in H_n, \, \, \gamma \in \SL_2(\Z/n\Z).
\end{equation*} 
But to say that $s_\alpha$ is a section for $\alpha$
compatible with $s_\beta$ means
\begin{equation}
\label{salpha}
s_\alpha(\gamma) \, h \, s_\alpha(\gamma)^{-1} = 
s_\beta (\gamma) h  \quad \text{ for all } h \in H_n, \, \, 
\gamma \in \SL_2(\Z/n\Z).
\end{equation}
The correspondence between (i) and (ii) is clear. Now given $s_\alpha$ 
compatible with $s_\beta$ we define 
$\theta^+(h,\gamma)=h s_\alpha(\gamma)$
and check using~(\ref{salpha}) that $\theta^+$ is a homomorphism.
Conversely, given $\theta^+$ we set $s_\alpha(\gamma)=\theta^+(1,\gamma)$.
This gives the correspondence between (i) and (iii).
\end{proof}

From the final statement of Theorem~\ref{thm:rho} we
immediately deduce

\begin{Theorem}
\label{thm:exheis}
The Schr\"odinger representation $\theta:H_n \to \SL_n(K)$ extends
uniquely to a representation $\theta^+:H_n^+ \to \SL_n(K)$.
Moreover the normaliser of $\theta(H_n)$ in $\SL_n(K)$ is 
$\theta^+(H_n^+)$.
\end{Theorem}

\begin{Remark}
The Schr\"odinger representation has $\phi(n)$ conjugates
obtained by either changing our choice of $\zeta_n$ or precomposing
with an automorphism of $H_n$. We may apply Theorem~\ref{thm:exheis}
to any one of these representations. 
This is important for our applications
and explains why we were careful to define $H_n^+$ before
introducing the Schr\"odinger representation. 
\end{Remark}

\section{Discrete covariants}
\label{sec:discov}

In this section we work over an algebraically closed field of
characteristic not equal to $2$, $3$ or $5$. The Hesse family 
of elliptic normal quintics (studied for example in 
\cite{g1hess}, \cite{Hu}) is given by
\begin{equation*}
\begin{array}{lccl}
u:& \Aff^2 & \to & \wedge^2 V \otimes W \\
& (a,b) & \mapsto & a \sum (v_{1} \wedge v_{4}) w_0 +
b  \sum (v_{2} \wedge v_{3}) w_0
\end{array} 
\end{equation*}
where the sums are taken over all cyclic permutations of the
subscripts mod $5$. The models $u(a,b)$, called the {\em Hesse models},
are representative of all genus one models in the following sense.

\begin{Lemma}
\label{lem:hessorb}
Every non-singular genus one model is $\GL(V) \times \GL(W)$-equivalent to
a Hesse model.
\end{Lemma}
\begin{Proof}
See \cite[Proposition 4.1]{g1hess}
\end{Proof}

The Hesse models are invariant under the following actions of the
Heisenberg group $H_5$ on $V$ and $W$.
\begin{equation}
\label{thetavw}
\begin{array}{lll} \smallskip
\theta_V: H_5 \to \SL(V)\,; & \sigma : v_i \mapsto \zeta_5^{2i} v_i\,;
& \tau : v_i \mapsto v_{i+1} \\
\theta_W: H_5 \to \SL(W)\,; & \sigma : w_i \mapsto \zeta_5^{i} w_i\,;
&  \tau : w_i \mapsto w_{i+1}.
\end{array} 
\end{equation}
Our definition of the Hesse family differs from that in 
\cite[Section 4]{g1hess} by a change of co-ordinates. This is 
to make the formulae~(\ref{thetavw}) more transparent than those
immediately preceding \cite[Lemma 7.7]{g1hess}.

Since $\theta_V$ and $\theta_W$ are conjugates of the 
Schr\"odinger representation they extend by Theorem~\ref{thm:exheis}
to representations of $H_5^+$. By abuse of notation we continue to 
write these representations as $\theta_V$ and $\theta_W$.
Let $Y$ be a rational representation of $\GL(V) \times \GL(W)$.
Then $\theta_V$ and $\theta_W$ 
define an action $\theta_Y$ of $H_5^+$ on $Y$. We write $Y^{H_5}$
for the subspace of $Y$ fixed by $H_5$. 
Since $H_5^+$ sits in an exact sequence
\begin{equation*}
 0 \ra H_5 \ra H_5^+ \ra \Gamma \ra 0 
\end{equation*}
there is an action of $\Gamma = \SL_2(\Z/5\Z)$ on $Y^{H_5}$. In the 
case $Y = \wedge^2 V \otimes W$ 
we find (by using Lemma~\ref{heisenchars} below to 
compute $\dim Y^{H_5}$) that 
${\rm Im}(u) = (\wedge^2 V \otimes W)^{H_5}$. 
The action of $\Gamma$ is then described by a representation 
$\chi_1 : \Gamma \to \GL_2(K)$ with the defining property that
\begin{equation}
\label{defchi}
 u  \circ \chi_1(\gamma) = \theta_{\wedge^2 V \otimes W}(\gamma) \circ u
\end{equation}
for all $\gamma \in \Gamma$.

\begin{Definition}
Let $\pi : \Gamma \to \GL(Z)$ be a representation. A 
{\em discrete covariant} (for $Z$) is a polynomial map $f : \Aff^2 \to Z$ 
satisfying
\[ u  \circ \chi_1(\gamma) = \pi (\gamma) \circ u\]
for all $\gamma \in \Gamma$.
\end{Definition} 

\begin{Theorem} 
\label{thm:dcov}
Let $F : \wedge^2 V \otimes W \to Y$ be a covariant. Then 
$f = F \circ u : \Aff^2 \to Y^{H_5}$ is a discrete covariant. Moreover
$F$ is uniquely determined by $f$. 
\end{Theorem}
\begin{Proof}
Since $F$ is a covariant it is 
$\SL(V) \times \SL(W)$-equivariant (by definition) 
and therefore $H_5^+$-equivariant. So its restriction to 
$(\wedge^2 V \otimes W)^{H_5}$ takes values in $Y^{H_5}$ and this
restriction is $\Gamma$-equivariant.

If $F_1$ and $F_2$ restrict to the same discrete covariant $f$ then
by Lemma~\ref{lem:hessorb} they agree on all non-singular models.
By Theorem~\ref{thm:inv} the non-singular models are Zariski dense in 
$\wedge^2 V \otimes W$ and from this we deduce that $F_1=F_2$.
\end{Proof}

For any given $Y$ it is easy to compute the discrete covariants 
using invariant theory for the finite groups $H_5$ and $\Gamma$.
We say that a discrete covariant $f : \Aff^2 \to Y^{H_5}$ 
{\em is a covariant} if it arises from a 
covariant $F: \wedge^2 V \otimes W \to Y$ 
as described in Theorem~\ref{thm:dcov}. It is important to note 
that not every discrete covariant is a covariant. 
We give examples below.

We recall the character table of $H_p$ for $p$ an odd prime.
There are $p^2+p-1$ conjugacy classes with representatives 
$\zeta^i$ and $\sigma^j \tau^k$ for $i,j,k \in \Z/p\Z$ 
with $(j,k) \not= (0,0)$. There are $p^2$ one-dimensional characters 
indexed by $(r,s) \in (\Z/p\Z)^2$. The remaining $p-1$ irreducible
characters are conjugates of the Schr\"odinger representation.
These are indexed by $t \in (\Z/p\Z)^\times$.
\begin{equation*}
\begin{array}{c|cc}
& \zeta^i & \sigma^j \tau^k \\ \hline
\lambda_{r,s} & 1 & \zeta_p^{jr+ks} \\
\theta_t & p \zeta_p^{it} & 0 \\
\end{array}
\end{equation*}

The dual of $\theta_t$ is $\theta_{-t}$. From the character table we also
deduce
\begin{Lemma}
\label{heisenchars}
Let $t,t' \in (\Z/p\Z)^\times$. \\ \smallskip
(i) $\theta_t \otimes \theta_{t'} \isom \left\{ \begin{array}{ll}
p \theta_{t+t'} & \text{ if } t+t' \not\equiv 0 \pmod{p} \\
\bigoplus_{r,s} \lambda_{r,s} & \text{ if } t+t' \equiv 0 \pmod{p} 
\end{array} \right.$ \\ \smallskip
(ii) $\wedge^d \theta_t \isom \left\{ \begin{array}{ll}
\lambda_{0,0} \qquad & \text{ if } d=0 \text{ or } d=p  \\ 
 \frac{1}{p} \binom{p}{d} \theta_{dt} & \text{ if } 1 \le d \le p-1
\end{array} \right. $ 
\quad  \\
(iii) $S^d \theta_t \isom \left\{ \begin{array}{ll}
\frac{1}{p}  \binom{p+d-1}{d} \theta_{dt} & \text{ if } d \not\equiv 0 \pmod{p} \\ \lambda_{0,0} \oplus \frac{1}{p^2} ( \binom{p+d-1}{d}-1) 
\bigoplus_{r,s} \lambda_{r,s} & \text{ if } d \equiv 0 \pmod{p}. 
\end{array} \right.$    
\end{Lemma}

By (\ref{thetavw}) 
the representations $W, V, V^*, W^*$ are equivalent to 
$\theta_t$ for $t=1,2,3,4$. In the examples at the end of this section 
we use Lemma~\ref{heisenchars} to compute the dimension 
of $Y^{H_5}$ and then find a basis by inspection.

The representation $\chi_1 : \Gamma \to \GL_2(K)$ 
defined by~(\ref{defchi}) works out as 
\begin{equation*}
\begin{array}{lrcl}
\chi_1(S) : & (a,b) & \mapsto & (\varphi a+b,a-\varphi b)/(\zeta_5^4-\zeta_5) \\
\chi_1(T) : & (a,b) & \mapsto & (\zeta_5^2 a,\zeta_5^3 b). 
\end{array} 
\end{equation*}
where $\varphi=1+\zeta_5+\zeta_5^4$.
To fix our notation for the other irreducible characters
we recall the character table for $\Gamma = \SL_2(\Z/5\Z)$. 
In the first row we list the sizes of the conjugacy classes. 
The same symbols are used to denote both a representation and its
character. We have written $\phibar=1-\varphi$.

\[ \begin{array}{c|@{\quad}ccccccccc} 
& 1 & 1 & 20 & 20 & 30 & 12 & 12 & 12 & 12 \\
&  I  & -I & ST  & -ST & S & T & -T & T^2 & -T^2 \\ \hline
\psi_1 & 1 & 1 & 1 & 1 & 1 & 1 & 1 & 1 & 1 \\ 
\psi_2 & 4 & 4 & 1 & 1 & 0 & -1 & -1 & -1 & -1 \\ 
\psi_3 & 5 & 5 & -1 & -1 & 1 & 0 & 0 & 0 & 0 \\ 
\psi_4 & 3 & 3 & 0 & 0 & -1 & \varphi & \varphi & \phibar & \phibar \\ 
\psi_5 & 3 & 3 & 0 & 0 & -1 & \phibar & \phibar & \varphi & \varphi \\ \hline
\chi_1 & 2 & -2 & -1 &  \quad  1 & 0 & -\varphi & \varphi & -\phibar & \phibar \\
\chi_2 & 2 & -2 & -1 & 1 & 0 & -\phibar & \phibar & -\varphi & \varphi \\
\chi_3 & 4 & -4 & 1 & -1 & 0 & -1 & 1 & -1 & 1 \\
\chi_4 & 6 & -6 & 0 & 0 & 0 & 1 & -1 & 1 & -1 \\
\end{array} \]

\medskip

The discrete covariants in the case $Y=K$ is the trivial representation
form the ring of {\em discrete invariants} $R = K[a,b]^\Gamma$. 
This ring was already studied by Klein.
We noted in \cite[Section 3]{g1hess} that $R$ is generated by 
\begin{equation}
\label{discinv}
\begin{aligned}
D & = ab(a^{10}-11a^5 b^5-b^{10}) \\ 
c_4 & = 
  a^{20} + 228 a^{15} b^5 + 494 a^{10} b^{10} - 228 a^5 b^{15} + b^{20} \\
c_6 & = 
 -a^{30} + 522 a^{25} b^5 + 10005 a^{20} b^{10} 
+ 10005 a^{10} b^{20} - 522a^5 b^{25} - b^{30}
\end{aligned}
\end{equation}
subject only to the relation $c_4^3-c_6^2 = 1728 D^5$.
The discrete invariants $c_4$ and $c_6$ are the restrictions of the
invariants $c_4$ and $c_6$ in Theorem~\ref{thm:inv}.
Our use of the same notation for both a 
covariant and its restriction to the Hesse family should not cause 
any confusion in view of the uniqueness statement in Theorem~\ref{thm:dcov}.

For an arbitrary representation $\pi : \Gamma \to \GL_m(K)$
the discrete covariants form an $R$-module $M_\pi$. We write
$M_\pi = \oplus_{d \ge 0} M_{\pi, d}$ for the grading by degree. 
For any given $\pi$ and $d$ it is easy to compute a basis
for $M_{\pi, d}$ by linear algebra.

\medskip

\begin{Lemma}
\label{freeprop}
Let $\pi:\Gamma \to \GL_m(K)$ be a representation. Then
\begin{enumerate}
\item $M_\pi$ is a free $K[D,c_4]$-module of rank $2m$. 
\item $M_\pi$ is a free $K[D,c_6]$-module of rank $3m$. 
\item $M_\pi$ is a free $K[c_4,c_6]$-module of rank $5m$. 
\end{enumerate}
Moreover if $M_\pi(r) \subset M_\pi$ is the direct sum of the 
graded pieces $M_{\pi,d}$ with $d \equiv r \pmod{5}$, then 
$M_\pi(r)$ is a free $K[c_4,c_6]$-module of rank $m$.
\end{Lemma}

\begin{proof}
In \cite[Lemma 5.3]{g1hess} we showed that $M_{\chi_1}$ is a free
$K[c_4,c_6]$-module. Since the same method (recalled from \cite{Benson})
works in general it only remains to compute the ranks.
Let $\K = K(a,b)^\Gamma$ be the field of fractions of $R$. 
By the normal basis theorem the
$\K[\Gamma]$-module $K(a,b)$ is a copy of the regular representation.
So if $\Gamma$ acts on $Z = K^m$ via $\pi$ then 
\[ \K \otimes Z \isom ( K(a,b) \otimes Z )^\Gamma = 
\K \otimes M_\pi. \]
In particular $\dim_\K(\K \otimes M_\pi) =m $. Statements (i)-(iii) 
follow since
\begin{align*}
[\K : K(D,c_4) ] & =  2, & 
[\K : K(D,c_6) ] & =  3, &
[\K : K(c_4,c_6) ] & =  5. 
\end{align*}

Finally we observe that the $M_\pi(r)$ for $r \in \Z/5\Z$
are free $K[c_4,c_6]$-modules with ranks $m_r$ (say) adding up to $5m$. 
Multiplication by $D$ shows that $m_r \le m_{r+2}$ for all $r$. 
Therefore $m_0= \ldots = m_4 =m$ as required.
\end{proof}


The Hilbert series of $M_\pi$ can be computed using 
Molien's theorem:
\begin{equation}
\label{molien}
h(M_\pi,t) = \sum_{d=0}^\infty  (\dim M_{\pi,d}) t^d 
=  \frac{1}{|\Gamma|} \sum_{\gamma \in \Gamma} \frac{ \Tr \pi (\gamma) }
{1-\Tr \chi_1(\gamma) t +t^2 }.  
\end{equation}
For example taking $\pi=\chi_1$ we find
\begin{align*} 
h(M_{\chi_1},t) & =   \frac{t+t^{11}+t^{19}+t^{29}}{(1-t^{12})(1-t^{20})} \\ 
& =   
\frac{t+t^{11}+t^{19}+t^{21}+t^{29}+t^{39}}{(1-t^{12})(1-t^{30})} \\
& =   \frac{(t+t^{11})+(t^{13}+t^{23})+(t^{25}+t^{35})
+(t^{37}+t^{47})+(t^{19}+t^{29})}{(1-t^{20})(1-t^{30})}.
\end{align*}
The numerators of these three expression give the degrees of the
generators in each part of Lemma~\ref{freeprop}.

There are essentially two ways in which a discrete covariant 
can fail to be a covariant. The first is that the weights 
computed using~(\ref{wtformula}) might fail to be integers.
For example the discrete invariant $D$ has weights $(p,q)=(24/5,12/5)$ 
and so cannot be an invariant. (However Theorem~\ref{thm:inv} tells us
that $c_4$, $c_6$ and $\Delta = D^5$ are invariants.)
Likewise taking $Y=S^5 W$ the discrete covariant 
\[ (a,b) \mapsto 
ab \textstyle\sum w_0^5 - 5 b^2 \sum w_0^3 w_1 w_4 + 5 a^2 \sum w_0^3 w_2 w_3
-30 ab \prod w_0. \]
has weights $(p,q)=(4/5,-3/5)$ and so cannot be a covariant.
The second is that the discrete covariant $f$ might arise from a 
{\em fractional covariant} by which we mean an 
$\SL(V) \times \SL(W)$-equivariant rational map 
$F: \wedge^2V \otimes W \ratto Y$.
It can happen that $f$ is regular even when $F$ is not. 
For example decomposing $(S^{10}W)^{H_5}$ as a $\Gamma$-module 
we find it contains a copy of the trivial representation. So there is
a discrete covariant of degree $0$. But there are clearly no
covariants $\wedge^2 V \otimes W \to S^{10}W$ of degree $0$.

In~\cite{paperII} we prove that these are the only
two obstructions. More precisely we show
that if $f:\Aff^2 \to Y^{H_5}$ is an integer weight discrete covariant 
then $\Delta^k f$ is a covariant for some 
integer $k \ge 0$. Moreover we give a practical method for determining 
the least such $k$.

If $Y$ is homogeneous of degree $(r,s)$ and $Y^{H_5} \not= 0$ 
then the action of the centre of $H_5$ shows that 
$2r+s \equiv 0 \pmod{5}$. We see by~(\ref{wtformula}) that 
$p$ is an integer if and only if $q$ is an integer.
So the integer weight condition is just a congruence mod~$5$ on the
degree of a covariant. In particular Lemma~\ref{freeprop} 
shows that the $K[c_4,c_6]$-module 
of integer weight discrete covariants is a free module of rank 
$m= \dim Y^{H_5}$. 

In this article we are primarily concerned with the 
rational representations $Y$ in the following table. 
In each case Lemma~\ref{heisenchars} shows that $\dim Y^{H_5} = 2$ or $3$. 
We list a basis
for $Y^{H_5}$ (the sums are taken over all cyclic permutations of the
subscripts mod~$5$) followed by its character as a $\Gamma$-module.
In the final column we list the degrees of the generators for the 
$K[c_4,c_6]$-module of integer weight discrete covariants, as 
computed using Molien's theorem.

\begin{center} {\bf Table 4.6 } \end{center}
\[ \begin{array}{clcll}
Y & \multicolumn{1}{c}{\text{ basis for $Y^{H_5}$}} 
& \text{character} & \text{degrees} \\ \hline
\wedge^2 V \otimes W     & 
\sum(v_1 \wedge v_4) w_0,\,  \sum(v_2 \wedge v_3) w_0 & \chi_1 & 1,11 \\
V^* \otimes \wedge^2 W   &
\sum v^*_0 (w_1 \wedge w_4),\,  \sum v^*_0 (w_2 \wedge w_3)  
&  \chi_2 & 7,17 \\
V \otimes \wedge^2 W^*   & 
\sum v_0 (w^*_1 \wedge w^*_4),\,  \sum v_0 (w^*_2 \wedge w^*_3)
&  \chi_2 & 13,23 \\
\wedge^2 V^* \otimes W^* &
\sum(v^*_1 \wedge v^*_4) w^*_0,\,  \sum(v^*_2 \wedge v^*_3) w^*_0
&  \chi_1 & 19,29 \\
V^* \otimes S^2W         & 
\sum v^*_0 w_0^2,\,  \sum v^*_0 w_1w_4,\,  \sum v^*_0 w_2w_3 &  \psi_4 & 2,12,22 \\
S^2V^* \otimes W^*       &
\sum v^{*2}_0 w^*_0,\,  \sum v^*_1 v^*_4 w^*_0,\,  \sum v^*_2 v^*_3 w^*_0
&  \psi_5 & 14,24,34 \\
S^2V \otimes W           
&\sum v_0^2 w_0,\,  \sum v_1 v_4 w_0,\,  \sum v_2 v_3 w_0
&  \psi_5 & 6,16,26 \\
V \otimes S^2W^*         
&\sum v_0 w^{*2}_0,\,  \sum v_0 w^*_1w^*_4,\,  \sum v_0 w^*_2w^*_3  
&  \psi_4 & 18,28,38 
\end{array} \] 
Checking the conditions in \cite{paperII} it turns out that 
each of these discrete covariants is a covariant. In~\cite{g1hess} we 
gave an alternative proof in the cases $Y= \wedge^2 V \otimes W$ 
and $Y= \wedge^2 V^* \otimes W^*$ using evectants.
The explicit constructions in Section~\ref{sec:constr} also 
show that each of these covariants exists at least 
as a fractional covariant.

The discrete covariant of degree $2$ for $Y = V^* \otimes S^2 W$
is $P_2 = \sum v_i^* p_i$ where 
\begin{equation}
\label{eqn:pi}
  p_i = ab w_i^2 + b^2 w_{i-1} w_{i+1} - a^2 w_{i-2} w_{i+2} 
\end{equation}
and the discrete covariant of degree $6$ for $Y = S^2V \otimes W$ is
\[ Q_6 = \textstyle\sum  
   (5 a^3 b^3 v_0^2 + a(a^5-3b^5) v_1 v_4 - b(3 a^5+b^5) v_2 v_3)w_0. \]
Substituting $v_i = p_i$ in $Q_6$ gives a covariant
of degree $10$ for $Y = S^5 W$ which turns out to be (a scalar multiple
of) the secant variety covariant~(\ref{def:S10}).
This suggested to us the algorithm for computing $Q_6$ 
in \cite[Section 8]{g1inv} that is the key step in our algorithm 
for computing the invariants.

In the remainder of this article we are concerned with arithmetic
applications of the covariants in Table~4.6 and in algorithms
for evaluating them on (non-singular) genus one models.

\section{Families of $5$-congruent elliptic curves}
\label{sec:congr}

From now on $K$ will be a field of 
characteristic $0$ with algebraic closure $\Kbar$. 

\begin{Definition}
(i) Elliptic curves $E$ and $E'$ over $K$ are $n$-congruent
if  $E[n]$ and $E'[n]$ are isomorphic as Galois modules. \\
(ii) The modular curve $Y_E^{(r)}(n) = 
X_E^{(r)}(n) \setminus \{ \text{cusps} \}$ parametrises
the family of elliptic curves $n$-congruent to $E$ via an isomorphism
$\psi$ with $e_n(\psi S,\psi T) = e_n(S,T)^{r}$ for all $S,T \in E[n]$. 
\end{Definition}

The curves $X_E^{(r)}(n)$ depend only on the class of $r \in 
(\Z/n\Z)^\times$ modulo squares. In the cases $r=\pm 1$ we denote them
$X_E(n)$ and $X_E^{-}(n)$. Rubin and Silverberg \cite{RubSil1}, 
\cite{RubSil2}, \cite{Silverberg} computed formulae 
for the families of elliptic curves parametrised by
$Y_E(n)$ for $n=2,3,4,5$. In \cite{g1hess} we gave a 
new proof of their result and extended to $Y_E^{-}(n)$ for $n=3,4,5$. 
In the case $n=5$ this is not so interesting since $-1$ is a 
square mod $5$. In Theorem~\ref{thm:ye2} below we remedy this by giving a 
formula for the family of elliptic curves parametrised by $Y_E^{(2)}(5)$.

First we need some preliminaries on Heisenberg groups. Since we have
dropped our earlier assumption that $K$ is algebraically closed our point
of view is slightly different from that in Section~\ref{sec:heis}.

\begin{Definition}
A {\em Heisenberg group} is a 
Galois invariant subgroup $H \subset \SL_n(\Kbar)$ such that
\begin{enumerate}
\item $H$ is the inverse image of a subgroup 
$\Delta \subset \PGL_n(\Kbar)$ with $\Delta \isom (\Z/n\Z)^2$. 
\item Taking commutators in $H$ induces a non-degenerate pairing 
$\Delta \times \Delta \to \mu_n$. 
\end{enumerate}
\end{Definition}

Let $C \subset \PP^{n-1}$ be a genus one normal curve of degree $n$.
The Heisenberg group defined by $C$ is the group of 
all matrices in $\SL_n(\Kbar)$
that act on $C$ as translation by an $n$-torsion point of its Jacobian $E$.
In this case the commutator pairing is the Weil pairing
$e_n : E[n] \times E[n] \to \mu_n$. If $C$ is a curve of degree $n=5$ 
then there is another
Heisenberg group determined by $C$ coming instead from the
action of $E[5]$ on the space of quadrics defining $C$. 

\begin{Lemma}
Let $\phi \in \wedge^2 V \otimes W$ be non-singular and 
let $E = \Jac(C_\phi)$.
Then there are projective representations $\chi_V : E[5] \to \PGL(V)$
and $\chi_W : E[5] \to \PGL(W)$ such that
\begin{enumerate}
\item The action of $E[5]$ on $C_\phi \subset \PP(W^*)$ 
is given by $\chi_W$ and 
\item $(\chi_V(T), \chi_W(T)) \phi \propto \phi$ for all $T \in E[5]$.
\end{enumerate}
\end{Lemma}
\begin{Proof}
If $g_W \in \GL(W)$ describes an automorphism of $C_\phi$
then by \cite[Lemma~7.6]{g1hess} there exists $g_V \in \GL(V)$,
unique up to sign, such that $(g_V,g_W) \phi = \phi$. 
So once we have used (i) to define $\chi_W$, 
condition (ii) uniquely determines $\chi_V$.
\end{Proof}

The projective representations $\chi_V$ and $\chi_W$ determine
Heisenberg groups
\begin{equation}
\label{def:H}
\begin{aligned}
H_1 & \subset \SL(W^*) & H_2 & \subset \SL(V^*) &
H_3 & \subset \SL(V) & H_4 & \subset \SL(W) 
\end{aligned}
\end{equation}
where the first of these is the Heisenberg group defined by $C_\phi$.
It follows by~(\ref{thetavw}) that 
the commutator pairing on $E[5]$ induced by $H_r$ is the $r$th 
power of the Weil pairing.

\begin{Theorem}
\label{thm:hps}
Let $\phi \in \wedge^2 V \otimes W$ be a non-singular 
genus one model determining
Heisenberg groups $H_1, \ldots, H_4$ as above.
Then the genus one normal curves with Heisenberg group $H_r$ are the
$C_{\phi'}$ for $\phi'$ a non-singular member of the pencil spanned
by $F_1(\phi)$ and $F_2(\phi)$ where $F_1$ and $F_2$ are a basis
for the $K[c_4,c_6]$-module of covariants $\wedge^2 V \otimes W \to Y$ and
\[ Y = \left\{ \begin{array}{ll} 
\wedge^2 V \otimes W & \text { if } r = 1 \\
V \otimes \wedge^2 W^* & \text { if } r = 2 \\
V^* \otimes \wedge^2 W & \text { if } r = 3 \\
\wedge^2 V^* \otimes W^* & \text { if } r = 4.
\end{array} \right. \] 
\end{Theorem}
\begin{Proof}
This generalises \cite[Theorem 8.2]{g1hess} where we treated the case
$r=1$. 

For the proof we may assume that $K$ is algebraically closed.
By Lemma~\ref{lem:hessorb} and
the covariance of $F_1$ and $F_2$ we may assume that
$\phi = u(a,b)$ is a Hesse model. Then each $H_r$ is the standard
Heisenberg group generated by the matrices~(\ref{heismats}). By 
\cite[Lemma 7.5]{g1hess} the genus one normal curves with this 
Heisenberg group are the $C_{\phi'}$ for $\phi'$ a non-singular Hesse
model. Splitting into the cases $r=1,2,3,4$ we checked by computing the
discrete covariants (see Remark~\ref{rem:indep} below for the case $r=3$) 
that $F_1(\phi)$ and $F_2(\phi)$ are linearly independent. 
They therefore span the space of Hesse models.
\end{Proof}

In \cite{g1hess} we studied the cases $r = 1,4$. 
We now work out explicit formulae in the case $r =3$. According to the
table at the end of Section~\ref{sec:discov} the $K[c_4,c_6]$-module 
of covariants for 
$Y = V^* \otimes \wedge^2 W$ is generated by covariants 
$\Psi_7$ and $\Psi_{17}$ of degrees $7$ and $17$. The corresponding 
discrete covariants are
\begin{equation}
\label{def:Psi}
(a,b) \mapsto f_d(a,b) \sum v_0^* (w_1 \wedge w_4) 
+ g_d(a,b) \sum v_0^* (w_2 \wedge w_3) 
\end{equation}
where
\begin{align*}
f_7(a,b)&  =  b^2 (7 a^5  - b^5), &  f_{17}(a,b) & = 
    b^2( 17 a^{15} + 187 a^{10} b^5 + 119 a^5 b^{10} + b^{15}), \\
g_7(a,b)&  =  a^2(a^5 + 7 b^5),  &  g_{17}(a,b) & = 
  -a^2(a^{15} - 119 a^{10} b^5 + 187 a^5 b^{10} - 17 b^{15}).
\end{align*}

\begin{Remark}
\label{rem:indep}
Direct calculation shows that 
$f_{7} g_{17} - g_{7} f_{17} =  -24 D^2$. We deduce that if $\phi$
is non-singular then $\Psi_7(\phi)$ and $\Psi_{17}(\phi)$ are linearly 
independent. In \cite{paperII} we generalise this to 
arbitrary $Y$.
\end{Remark}

We recall that the ring of discrete invariants $K[a,b]^\Gamma$ is 
generated by the polynomials $D$, $c_4$ and $c_6$ in (\ref{discinv}).
\begin{Lemma}
\label{lem:hp}
There are polynomials $\D(\la,\mu)$, $\cc_4(\la,\mu)$ and $\cc_6(\la,\mu)$
with coefficients in $K[c_4,c_6]$ such that
\begin{align*}
 \D(\la,\mu) &= 27 \cdot 
D( \la f_7 + \mu f_{17}, \la g_7 + \mu g_{17} )/D(a,b)^2 \\
 \cc_4(\la,\mu) &= 54^2 \cdot 
c_4( \la f_7 + \mu f_{17}, \la g_7 + \mu g_{17} ) \\ 
 \cc_6(\la,\mu) &= 54^3 \cdot 
c_6( \la f_7 + \mu f_{17}, \la g_7 + \mu g_{17} ) 
\end{align*}
\end{Lemma}
\begin{Proof} The coefficients are 
discrete invariants of degree a multiple of $5$. We can then
therefore write them as polynomials in $c_4$ and $c_6$. 
(The factors $27$, $54^2$, $54^3$ are included to make $\D$, $\cc_4$, 
$\cc_6$ primitive polynomials in $\Z[c_4,c_6,\la,\mu]$.)
\end{Proof}

The polynomials $\D(\la,\mu)$, $\cc_4(\la,\mu)$ and $\cc_6(\la,\mu)$
are easily computed from the description in Lemma~\ref{lem:hp}. We find
\begin{align*}
\D(& \la, \mu) =  
-(125 c_4^3 + 64 c_6^2) \la^{12} - 1620 c_4^2 c_6 \la^{11} \mu 
- 66 (25 c_4^4 + 56 c_4 c_6^2) \la^{10} \mu^2 \\ & -
    220(11 c_4^3 c_6 + 16 c_6^3) \la^9 \mu^3 
 +  1485(5 c_4^5 + 4 c_4^2 c_6^2) \la^8 \mu^4 + 
   792 (53 c_4^4 c_6 + 28 c_4 c_6^3) \la^7 \mu^5 
\\ & + 660 (9 c_4^6 + 164 c_4^3 c_6^2 + 16 c_6^4) \la^6 \mu^6  + 
  2376(19 c_4^5 c_6 + 44 c_4^2 c_6^3) \la^5 \mu^7 \\ & + 495 (27 c_4^7 + 
   104 c_4^4 c_6^2 +  112 c_4 c_6^4) \la^4 \mu^8 
   +  220 (81 c_4^6 c_6 + 136 c_4^3 c_6^3 +  80 c_6^5) \la^3 \mu^9 
  \\ &  - 594 (9 c_4^8 - 32 c_4^5 c_6^2 - 16 c_4^2 c_6^4) \la^2 \mu^{10} - 
    60(135 c_4^7 c_6 - 328 c_4^4 c_6^3 + 112 c_4 c_6^5) \la \mu^{11} \\ &
  - (729 c_4^9 + 108 c_4^6 c_6^2 - 2896 c_4^3 c_6^4 + 1600 c_6^6) \mu^{12},
\end{align*}
\[ \cc_4 (\lambda,\mu) =  
\frac{-1}{11^2 \cdot 12^2} \left| \begin{matrix} \smallskip
\frac{\partial^2 \D}{\partial \lambda^2} & 
\frac{\partial^2 \D}{\partial \lambda \partial \mu} \\
\frac{\partial^2 \D}{\partial \lambda \partial \mu} &
\frac{\partial^2 \D}{\partial \mu^2} 
\end{matrix} \right| 
\quad \text{ and } \quad 
 \cc_6(\lambda,\mu)  =  
\frac{-1}{12 \cdot 20} \left| \begin{matrix} \smallskip
\frac{\partial \D}{\partial \lambda} & 
\frac{\partial \D}{\partial \mu} \\
\frac{\partial \cc_4}{\partial \lambda} &
\frac{\partial \cc_4}{\partial \mu}
\end{matrix} \right|. \]
These polynomials satisfy the relation
\begin{equation*}
\label{syz}
 \cc_4(\lambda,\mu)^3- \cc_6(\lambda,\mu)^2 
= (c_4^3-c_6^2)^2 \, \D(\lambda,\mu)^5. 
\end{equation*}
We have contributed them to Magma \cite{magma} as
{\tt HessePolynomials(5,2,[c4,c6])}.

\begin{Lemma}
\label{lem:hps}
The $K[c_4,c_6]$-module of covariants $\wedge^2 V \otimes W 
\to V^* \otimes \wedge^2 W$ is generated by covariants
$\Psi_{7}$ and $\Psi_{17}$ satisfying
\begin{align*}
c_4( \la \Psi_7 + \mu \Psi_{17} ) &= \cc_4(\la,\mu)/54^2 \\
c_6( \la \Psi_7 + \mu \Psi_{17} ) &= \cc_6(\la,\mu)/54^3. 
\end{align*}
\end{Lemma}
\begin{Proof} By Lemma~\ref{lem:hessorb} and the covariance of 
$\Psi_{7}$ and $\Psi_{17}$ it suffices 
to check these identities on the Hesse family. But in that case 
we are done by the definitions of $\cc_4$ and $\cc_6$ in Lemma~\ref{lem:hp}.
\end{Proof}

We say that elliptic curves $E$ and $E'$ are {\em indirectly $5$-congruent}
if there is an isomorphism of Galois modules $\psi : E[5] \isom E'[5]$ 
with $e_5(\psi S,\psi T) = e_5(S,T)^r$ for some $r \in \{2,3\}$.
In the notation introduced at the start of this section the
elliptic curves indirectly $5$-congruent to $E$ are parametrised by
$Y_E^{(2)}(5)$.

\begin{Theorem} 
\label{thm:ye2}
Let $E$ be an elliptic curve over $K$ with Weierstrass
equation
\[ y^2 = x^3 - 27 c_4 x - 54 c_6.\]
Then the family of elliptic curves parametrised by $Y_E^{(2)}(5)$ is 
\[ E_{\lambda,\mu} : \quad 
  y^2 = x^3 - 12 \cc_4(\lambda,\mu) x - 16 \cc_6(\lambda,\mu)\]
where the coefficients of
$\cc_4(\lambda,\mu)$ and $\cc_6(\lambda,\mu)$ 
are evaluated at $c_4,c_6 \in K$. 
\end{Theorem}
\begin{Proof}
We embed $E \subset \PP^{4}$ via the complete linear system $|5.0_E|$. 
The image is defined by some $\phi \in \wedge^2 V \otimes W$ 
with invariants $c_4$ and $c_6$. 
Let $H_1, \ldots, H_4$ be the Heisenberg groups~(\ref{def:H}) 
determined by $\phi$. \\ 
(i) Suppose that $\phi' = \lambda \Psi_7(\phi) + \mu \Psi_{17}(\phi)$ 
is non-singular. By Theorem~\ref{thm:hps} the genus one normal curves
$C_\phi \subset \PP^4$ and $C_{\phi'} \subset \PP^4$ have Heisenberg
groups $H_1$ and $H_3$. By Theorem~\ref{thm:inv} and Lemma~\ref{lem:hps}
the Jacobians of these curves are $E$ and $E_{\lambda,\mu}$. Since the
Heisenberg group carries the information of both the action of Galois
on the $5$-torsion of the Jacobian, and the Weil pairing (via the commutator),
it follows that $E$ and $E_{\lambda,\mu}$ are indirectly $5$-congruent. \\
(ii) Let $E'$ be an elliptic curve indirectly $5$-congruent to $E$. 
By \cite[Theorem 5.2]{paperI} there is a genus one normal curve 
$C' \subset \PP^{4}$ with Jacobian $E'$ and Heisenberg
group $H_3$. Then Theorem~\ref{thm:hps} shows that $C' = C_{\phi'}$ 
for some $\phi' = \lambda \Psi_7(\phi) + \mu \Psi_{17}(\phi)$.
Taking Jacobians gives $E' \isom E_{\lambda,\mu}$.
\end{Proof}

We also worked out formulae corresponding to the case $r = 2$ of
Theorem~\ref{thm:hps}. 
We omit the details since the family of elliptic curves obtained is 
the same as that in Theorem~\ref{thm:ye2}. (We encountered a similar
situation in \cite{g1hess} with the cases $r = \pm 1$.)

\begin{Example}
\label{ex:congr}
Let $F/\Q$ be the elliptic curve $y^2 + x y = x^3 - 607 x + 5721$
labelled $2834c1$ in Cremona's tables \cite{CrTables}. The invariants
of this Weierstrass equation are $c_4= 29137$ and $c_6 = -4986649$.
Substituting into the above expression for $\D$
and then making a change of variables\footnote{This change of variables 
was found by minimising to make the numerical factor on the right 
hand side of~(\ref{minrel}) a small integer, and reducing 
as described in \cite{SC}.} 
to simplify we obtain
\begin{align*}
{\mathfrak D}(\xi ,\eta) & = \frac{1}{2^{89} \cdot 3^{15} \cdot 13^{9}}
\D( 1663 \xi + 2850 \eta, 7 \xi + 18 \eta) \\
& = -60647 \xi^{12} + 74183 \xi^{11} \eta - 366344 \xi^{10} \eta^2
   - 965800 \xi^9 \eta^3 \\  
& ~\quad{}+ 1640430 \xi^8 \eta^4 - 166188 \xi^7 \eta^5 + 1473362 \xi^6 \eta^6 
   - 1041216 \xi^5 \eta^7 \\ 
& ~\quad{}+ 1224300 \xi^4 \eta^8 + 816860 \xi^3 \eta^9 
   - 474188 \xi^2 \eta^{10} + 22692 \xi \eta^{11} - 51256 \eta^{12}.
\end{align*}

By Theorem~\ref{thm:ye2} the family of elliptic curves indirectly 
$5$-congruent to $F$ is 
$y^2 = x^3 - 27 {\mathfrak c}_4(\xi,\eta) x - 54 {\mathfrak c}_6(\xi,\eta)$
where
\[ {\mathfrak c}_4 (\xi,\eta) =  
\frac{\!\!-1\,\,\,}{11^2} \left| \begin{matrix} \smallskip
\frac{\partial^2 {\mathfrak D}}{\partial \xi^2} & 
\frac{\partial^2 {\mathfrak D}}{\partial \xi \partial \eta} \\
\frac{\partial^2 {\mathfrak D}}{\partial \xi \partial \eta} &
\frac{\partial^2 {\mathfrak D}}{\partial \eta^2} 
\end{matrix} \right| 
\quad \text{ and } \quad 
 {\mathfrak c}_6(\xi,\eta)  =  
\frac{\!\!-1\,\,}{20} \left| \begin{matrix} \smallskip
\frac{\partial {\mathfrak D}}{\partial \xi} & 
\frac{\partial {\mathfrak D}}{\partial \eta} \\
\frac{\partial {\mathfrak c}_4}{\partial \xi} &
\frac{\partial {\mathfrak c}_4}{\partial \eta}
\end{matrix} \right|. \]
These polynomials satisfy the relation
\begin{equation}
\label{minrel}
{\mathfrak c}_4(\xi,\eta)^3- {\mathfrak c}_6(\xi,\eta)^2 
= 2 \cdot 13 \cdot 109^2 \cdot 1728 \, {\mathfrak D}(\xi,\eta)^5.
\end{equation}

We specialise $\xi, \eta$ to integers with $\max(|\xi|,|\eta|) \le 100$
and sort by conductor to obtain a list of elliptic curves that begins
\[ \begin{array}{cc|rl}
\xi & \eta & \multicolumn{1}{c}{\text{conductor}} & \multicolumn{1}{c}{[a_1, \ldots, a_6]}
\\ \hline
 3& 2 & 2834 &[ 1, -1, 1, -8109, -279017 ] \\
 0& 1 & 18157438 &[ 1, -1, 1, 68377761, 119969009527 ] \\
1& 0 & 171873598 &[ 1, 0, 0, 895245563, 21917334070263 ] \\
 2& 1 & 205326134 &[ 1, 0, 0, -637387852699482, -6550975667615204649116 ] \\
 1& 1 & 1506404198 &[ 1, 0, 0, -793652608607, -207340288851298727 ] \\
 1& -1 & 6582143542 &[ 1, 0, 0, -2705846635122, -1178369764561303100 ] 
\end{array} \]
The first curve in this list is the elliptic curve $E$ labelled
$2834d1$ in Cremona's tables. We discuss the elliptic curves $E$ and $F$ 
further in Example~\ref{ex:vis}.

\end{Example}

\section{Doubling in the $5$-Selmer group and visibility}
\label{sec:double}

Let $E/K$ be an elliptic curve. In \cite{paperI} we 
interpreted the group $H^1(K,E[n])$
as parametrising Brauer-Severi diagrams 
$[C \to S]$ as twists of $[ E \to \PP^{n-1}]$.
We also studied the obstruction map $\Ob_n : H^1(K,E[n]) \to \Br(K)[n]$ 
that sends the class of $[C \to S]$ to the class of the Brauer-Severi 
variety $S$. The diagrams with trivial obstruction, i.e. $S \isom \PP^{n-1}$, 
are genus one normal curves of degree~$n$. Strictly speaking a
diagram includes the choice of an action of $E$ on $C$. So in general 
a genus one normal curve of degree $n$ with Jacobian $E$ 
represents a pair of inverse elements in $H^1(K,E[n])$.

In the case $n=5$ we obtain the following partial interpretation
of $H^1(K,E[5])$ in terms of genus one models. We say that
genus one models are {\em properly equivalent} if they
are related by $(g_V,g_W) \in \GL(V) \times 
\GL(W)$ with $(\det g_V)^2 \det g_W = 1$. 

\begin{Theorem}
\label{thm:kerob}
Let $c_4$ and $c_6$ be the invariants of a Weierstrass equation
for~$E$. Then the genus one models over $K$ with invariants $c_4$
and $c_6$, up to proper $K$-equivalence, are parametrised by
$\ker(\Ob_5) \subset H^1(K,E[5])$. 
\end{Theorem}
\begin{Proof}
This is analogous to the case $n=3$ treated in 
\cite[Theorem 2.5]{testeq}. The proof given there relies on a 
statement about invariant differentials which is generalised to
the case $n=5$ in \cite[Proposition 5.19]{g1inv}.
\end{Proof}

The obstruction map is not a group 
homomorphism and its kernel is not a group. So given two genus 
one models with the same invariants, 
their sum in $H^1(K,E[5])$ need not be represented by a 
genus one model. However the obstruction map is quadratic and in
particular satisfies $\Ob_n( a \xi) = a^2 \Ob_n(\xi)$ for $a \in \Z$.
If $\phi$ is a genus one model representing $\xi \in \ker(\Ob_5)$ then 
$-\phi$ (which has the same invariants) represents $-\xi$. In this 
section we show how to find a model $\phi'$ representing $\pm 2 \xi$.
It turns out that $C_{\phi'}$ sits inside an ambient space $\PP^4$  
that is naturally the dual of the ambient space for $C_\phi$.

First we need to recall another of the interpretations of
$H^1(K,E[n])$ given in \cite{paperI}. A {\em theta group} for $E[n]$ 
is a central extension of $E[n]$ by $\Gm$ with commutator given
by the Weil pairing. The base theta group $\Theta_E \subset \GL_n(\Kbar)$ 
is the set of all matrices that act on the base diagram 
$[E \to \PP^{n-1}]$ as translation by an $n$-torsion point of $E$. 
Then $H^1(K,E[n])$
parametrises the theta groups for $E[n]$ as twists of $\Theta_E$.

In Section~\ref{sec:congr} we saw that a non-singular genus one model
$\phi$ determines Heisenberg groups $H_1, \ldots, H_4$. We write 
$\Theta_r$ for the theta group generated by $H_r$ and the scalar 
matrices. Writing $E = \Jac(C_\phi)$ we see that $\Theta_r$ is a
theta group for $E[5]$ where the latter is equipped with the
$r$th power of the Weil pairing.

\begin{Lemma}
\label{lem:theta}
Let $\phi, \phi' \in \wedge^2 V \otimes W$ be non-singular
genus one models determining theta groups $\Theta_1, \ldots, \Theta_4$
and $\Theta'_1, \ldots, \Theta'_4$. If $\Jac(C_\phi)=\Jac(C_{\phi'})=E$
then there exists $\xi \in H^1(K,E[5])$ and isomorphisms 
$\gamma_r : \Theta_r \isom \Theta'_r$ such that
\begin{equation}
\label{thetatwist}
  \sigma(\gamma_r) \gamma_r^{-1} : x \mapsto e_5(\xi_\sigma,x)^r x 
\end{equation}
for all $\sigma \in \Gal(\Kbar/K)$ and $r \in (\Z/5\Z)^\times$.
\end{Lemma}

\begin{Proof}
The curves $C_\phi$ and $C_{\phi'}$ are isomorphic over $\Kbar$.
So by \cite[Proposition 4.6]{g1inv} there exists $g = (g_V,g_W)
\in \GL(V) \times \GL(W)$ with $g \phi = \phi'$. In fact we 
can choose $g$ so that it induces the identity map
on the Jacobian $E$. The isomorphisms $\gamma_1, \ldots, \gamma_4$
are conjugation by $g_W^{-T}, g_V^{-T}, g_V, g_W$
where the superscript $-T$ indicates inverse transpose.
Then $\sigma(\gamma_r) \gamma_r^{-1}$ is conjugation by an element of 
$\Theta'_r$ above $\xi_\sigma \in E[5]$, with $\xi_\sigma$ independent of 
$r$. The conclusion~(\ref{thetatwist}) follows since the commutator
pairing for $\Theta_r$ is the $r$th power of the Weil pairing.
\end{Proof}

\begin{Theorem} 
\label{thm:theta}
Let $E$ and $F$ be elliptic curves over $K$ 
and $\psi: E[5] \isom F[5]$ an isomorphism of Galois modules 
with $e_5(\psi S, \psi T) = e_5(S,T)^r$
for some $r \in (\Z/5\Z)^\times$. Let $\Theta_1, \ldots, \Theta_4$ be 
the theta groups determined by a non-singular genus one model 
$\phi \in \wedge^2 V \otimes W$ with $\Jac(C_\phi) = E$. 
If $\Theta_1$ is the twist of $\Theta_E$ by $\xi \in H^1(K,E[5])$ then
$\Theta_r$ is the twist of $\Theta_F$ by $\psi_*(\xi) \in H^1(K,F[5])$
\end{Theorem}
\begin{Proof}
We first prove the case $\xi = 0$. We claim that if $\phi$ describes the
image of $E \subset \PP^4$ embedded by $|5.0_E|$ then 
$\Theta_r \to E[5]$ has a Galois
equivariant section $T \mapsto M_T$ with $M_S M_T = e_5(S,T)^{r/2} M_{S+T}$.
We recall the proof of this in the case $r=1$ from \cite[Lemma 3.11]{paperI}.
The $[-1]$-map on $E$ lifts to $\iota \in \PGL_5(K)$. Then there 
is a unique scaling of $M_T$ such that $M_T^5 = I$ and $\iota M_T \iota^{-1}
= M_T^{-1}$. The uniqueness ensures that  $T \mapsto M_T$ is 
Galois equivariant. Now if $M_S M_T = \lambda M_{S+T}$ then conjugating
by $\iota$ gives $M_S^{-1} M_T^{-1} = \lambda M_{S+T}^{-1}$ and so
$\lambda^2 = M_S M_T M_S^{-1} M_T^{-1} = e_5(S,T)$. This proves the
claim when $r=1$. The cases $r=2,3,4$ are similar using that the 
$[-1]$-map induces involutions in $\PGL(V)$ and $\PGL(W)$.
Applying the claim to both $E$ and $F$ we see that if 
$\Theta_1 \isom \Theta_E$ then $\Theta_r \isom \Theta_F$.
This proves the theorem in the case $\xi=0$. 
The general case follows by Lemma~\ref{lem:theta}.
\end{Proof}

According to Table~4.6 the $K[c_4,c_6]$-module 
of covariants for $Y = \wedge^2 V^* \otimes W^*$ is generated 
by covariants $\Pi_{19}$ and $\Pi_{29}$ of degrees $19$ and $29$. 
The corresponding discrete covariants are
\begin{equation}
\label{def:Pi}
 (a,b) \mapsto \textstyle\frac{1}{\deg c_k} \left(
  \frac{\partial c_k}{\partial a} \sum (v^*_1 \wedge v^*_4) w^*_0 
+ \frac{\partial c_k}{\partial b} \sum (v^*_2 \wedge v^*_3) w^*_0 
\right) 
\end{equation}
for $k=4,6$. As noted in \cite{g1hess} these are the evectants of 
$c_4$ and $c_6$.


\begin{Theorem}
\label{thm:double}
Let $\Pi_{49} = \frac{1}{144} (c_6 \Pi_{19} - c_4  \Pi_{29})$ be
the covariant of degree $49$ whose restriction to the Hesse family
is 
\[ (a,b) \mapsto  D^4 \big( b \textstyle\sum (v_1^* \wedge v_4^*) w_0^* - a
\sum (v_2^* \wedge v_3^*) w_0^* \big) \]
If $\phi \in \wedge^2 V \otimes W$ is non-singular and
$\phi' = \Pi_{49}(\phi)$ then  
\begin{enumerate}
\item $C_\phi$ and $C_{\phi'}$ have the same Jacobian elliptic curve $E$, and
\item the class of $[C_{\phi'} \to \PP^4]$ is twice the class of  
$[C_\phi \to \PP^4]$ in $H^1(K,E[5])$.
\end{enumerate}
\end{Theorem}
\begin{Proof} (i) By considering $\phi$ a Hesse model we deduce
\begin{align*}
c_4(\phi') &= \Delta(\phi)^{16} c_4(\phi) \\
c_6(\phi') &= \Delta(\phi)^{24} c_6(\phi) 
\end{align*}
It follows by Theorem~\ref{thm:inv} that the Jacobians are isomorphic. \\
(ii) We apply Theorem~\ref{thm:theta} in the case $\psi : E[5] \to E[5]$ is 
multiplication by $2$. This shows that the double of 
$[C_{\phi} \to \PP^4]$ has theta group $\Theta_4$. By (i) and 
Theorem~\ref{thm:hps} this double is $[C_{\phi'} \to \PP^4]$.
\end{Proof}

\begin{Remarks}
(i) Whether Theorem~\ref{thm:double} is a formula for doubling or 
tripling in $H^1(K,E[5])$ depends on the choice of isomorphism 
$\Jac(C_\phi) \isom \Jac(C_{\phi'})$. We have not attempted to resolve
these sign issues. \\
(ii) If $K$ is a number field then the $n$-Selmer group 
$S^{(n)}(E/K)$ is by definition a subgroup $H^1(K,E[n])$. It is well known
that  $S^{(n)}(E/K) \subset \ker (\Ob_n)$. Thus Theorem~\ref{thm:double}
gives a formula for doubling/tripling in the $5$-Selmer group. \\
(iii) Let $g =(g_V,g_W) \in \GL(V) \times \GL(W)$ be any element defined over
$K$ with $(\deg g_V)^2 (\det g_W) = \Delta(\phi)^4$, for example 
a pair of diagonal matrices. Then in terms of Theorem~\ref{thm:kerob}
the double/triple of $\phi$ is $\pm g^{-1} \Pi_{49} (\phi)$.
\end{Remarks}


\begin{Example} 
\label{ex:double}
Wuthrich \cite{W} constructed an element of order $5$
in the Tate-Shafarevich group of the elliptic curve $E/\Q$ with
Weierstrass equation
\[ y^2 + x y + y = x^3 + x^2 - 3146 x + 39049. \]
His example (also discussed in \cite[Section 9]{g1inv}) 
is defined by the $4 \times 4$ Pfaffians of
\[ \begin{pmatrix}
0  & 310 x_1 + 3 x_2 + 162 x_5 &  -34 x_1 - 5 x_2 - 14 x_5  
& 10 x_1 + 28 x_4 + 16 x_5 &  80 x_1 - 32 x_4 \\
&  0 &  6 x_1 + 3 x_2 + 2 x_5 &  -6 x_1 + 7 x_3 - 4 x_4 &  -14 x_2 - 8 x_3 \\
&  &  0 &  -x_3  & 2 x_2 \\
& - &  &  0  & -4 x_1 \\
&  &  &  &  0
\end{pmatrix} \]
The algorithms in \cite{minred5} suggest making a change of co-ordinates
\[ \begin{pmatrix} x_1 \\ x_2 \\ x_3 \\ x_4 \\ x_5 \end{pmatrix}
\leftarrow
\begin{pmatrix}
  0 &  4  & -8  &  4  &  8 \\
  0 &   0 &   0 &   0 &  16 \\
  0 &  -4  &  4  &  0  & 12 \\
  4 &   5 & -15  &  2  &  7\\
  4 & -12 &  20 & -12  & -8
\end{pmatrix} \begin{pmatrix} x_1 \\ x_2 \\ x_3 \\ x_4 \\ x_5 \end{pmatrix} \]
so that Wuthrich's example becomes
\[ \begin{pmatrix}
0 &  x_2 + x_5 &  -x_5  & -x_1 + x_2 &  x_4 \\
 &  0 &  x_2 - x_3 + x_4 &  x_1 + x_2 + x_3 - x_4 - x_5  & x_1 - x_2 - x_3 - x_4 - x_5 \\
 &  &  0 &  x_1 - x_2 + 2 x_3 - x_4 - x_5 &  -x_2 - x_4 + x_5 \\
 & - &  &  0 &  -x_3 - x_4 - 2 x_5 \\
 &  &  &  &  0
\end{pmatrix} \]
Our Magma function {\tt DoubleGenusOneModel} uses the algorithms 
in Section~\ref{sec:constr} to evaluate 
$\Pi_{19}$ and $\Pi_{29}$ and then returns
$\Pi_{49} = \frac{1}{144} (c_6 \Pi_{19} - c_4  \Pi_{29})$. 
Running it on the
above model $\phi$ gives a model $\phi'$ with entries
\small
\begin{align*}
\phi'_{12} &= 3534132778 x_1 + 3583651940 x_2 - 881947110 x_3 - 323014538 x_4 + 3395115339 x_5 \\ 
\phi'_{13} &= 5079379222 x_1 - 2965539950 x_2 + 11022202860 x_3 + 12821590868 x_4 + 640276471 x_5 \\ 
\phi'_{14} &= -10098238458 x_1 - 1274966110 x_2 - 7873816170 x_3 - 3456923272 x_4 - 62353929 x_5 \\ 
\phi'_{15} &= -12929747724 x_1 - 6790511810 x_2 - 11113305270 x_3 - 15161763156 x_4 + 3241937033 x_5 \\ 
\phi'_{23} &= -3381247332 x_1 + 3810679160 x_2 + 5919634530 x_3 + 75326852 x_4 - 1245085426 x_5 \\ 
\phi'_{24} &= -3572860258 x_1 - 5569480730 x_2 - 953739600 x_3 - 2138046812 x_4 - 858145244 x_5 \\ 
\phi'_{25} &= -4674149266 x_1 - 943631490 x_2 - 6754488160 x_3 + 751535046 x_4 + 117685567 x_5 \\ 
\phi'_{34} &= -1851228934 x_1 + 5238146110 x_2 - 165588410 x_3 - 2070411506 x_4 + 678105748 x_5 \\ 
\phi'_{35} &= -6992835070 x_1 - 3744630360 x_2 + 3130208220 x_3 - 4523781310 x_4 + 433739425 x_5 \\ 
\phi'_{45} &= 780078472 x_1 + 2039763820 x_2 - 450062790 x_3 - 7105731722 x_4 + 1625466111 x_5 
\end{align*}
\normalsize

The algorithms in \cite{minred5} suggest making a change of co-ordinates
\[ \begin{pmatrix} x_1 \\ x_2 \\ x_3 \\ x_4 \\ x_5 \end{pmatrix}
\leftarrow
\begin{pmatrix}
  92 & -36 &-153 & 129 &-131 \\
 -54 &  84 &   5 &-206 & 139 \\
 -63 &-174 & -60 & -79 &  53 \\
-111 & 106 & 206 &-115 &-162 \\
 314 &-466 & 158 &-328 & -12
\end{pmatrix} 
\begin{pmatrix} x_1 \\ x_2 \\ x_3 \\ x_4 \\ x_5 \end{pmatrix} \]
whereupon the model $\phi'$ simplifies to 
\[ \begin{pmatrix}
0  & -x_4 + x_5 &  x_3 - x_4 + x_5 &  x_2 - x_5 &  x_1 - x_2 + x_3 - x_4 - 2 x_5 \\
 &  0 &  x_1 + x_5  & -x_2 - x_3 &  -x_2 + x_5 \\
 &  &  0 &  x_4 &  -x_1 \\
 & - &  &  0 &  x_1 + x_4 - x_5 \\
 &  &  &  &  0
\end{pmatrix} \]
This is the double in $\Sha(E/\Q)[5]$ of Wuthrich's example. If we double 
again then we get back to the original example. Moreover the matrices 
needed to minimise and reduce are the transposes of those used above.
\end{Example}

Let $E$ and $F$ be a pair of $n$-congruent elliptic curves.
In terminology introduced by Mazur \cite{CM} the 
{\em visible subgroup} of $H^1(K,E)$ explained by $F(K)$ is the image 
of the composite
\[ \frac{F(K)}{nF(K)} \stackrel{\delta}{\ra} 
H^1(K,F[n]) \isom H^1(K,E[n]) \stackrel{\iota}{\to} H^1(K,E)[n] \]
where the maps $\delta$ and $\iota$ come from the
Kummer exact sequences for $E$ and $F$, and the middle isomorphism
is induced by the congruence. Our interest is 
in using visibility to compute explicit elements of 
$H^1(K,E)$. In \cite{g1hess} we gave examples in the cases 
$n=2,3,4,5$ assuming in the case $n=5$ that the congruence 
$E[5] \isom F[5]$ respects the Weil pairing. 
Using Theorems~\ref{thm:hps} and~\ref{thm:theta} and the explicit
constructions in Section~\ref{sec:constr} we may now remove
this restriction.

\begin{Example} 
\label{ex:vis}
We start with the pair of elliptic curves $E = 2834d1$ and $F = 2834c1$
taken from \cite[Table 1]{CM}. We have already seen in 
Example~\ref{ex:congr} that $E$ and $F$ are indirectly $5$-congruent.
Alternatively this may be checked as follows. Let 
$\cc_4(\la,\mu)$ and $\cc_6(\la,\mu)$ be the polynomials defined
in Section~\ref{sec:congr} with coefficients specialised to the invariants
$c_4= 29137$ and $c_6 = -4986649$ of $F$. 
Then writing $j_E =-389217^3/(2 \cdot 13 \cdot 109^3)$
for the $j$-invariant of $E$ we find that the binary form of degree 60
  \[ (1728-j_E ) \cc_4(\la,\mu)^3+j_E \cc_6(\la,\mu)^2 = 0 \]
has a unique $\Q$-rational root. Substituting this root
$(\la:\mu) = ( 3563: 19)$ into Theorem~\ref{thm:ye2} confirms that
$E$ and $F$ are indirectly $5$-congruent. 

Our method is now the same as that in \cite[Section 15]{g1hess} except 
that in place of the Hessian we use the covariants
\[  \Psi_{7}, \Psi_{17} : \wedge^2 V \otimes W \to V^* \otimes \wedge^2 W. \] 
We have $F(\Q) \isom \Z^2$ generated by $P_1 = (-10,109)$ and
$P_2 = (-28,45)$. If we embedding $F \subset \PP^4$ via the complete linear
system $|4.0_F + P|$ with $P = P_1$ then the image is defined by 
a genus one model $\phi$. Our Magma function {\tt GenusOneModel(5,P)} 
computes such a model 
\small
\[ \begin{pmatrix}
0 &  x_2 + 2 x_4 - 3 x_5  & 3 x_1 - x_2 + 8 x_3 + 2 x_4 - 3 x_5  & x_1 - 2 x_4 + 3 x_5 &  x_3 - x_4 + x_5 \\
 & 0 & 3 x_2 + 2 x_3 + 2 x_4 + 3 x_5 &  x_2 + x_3  & x_5 \\
 &  &  0 &  x_3 + 2 x_4 + 2 x_5  & x_4 + x_5 \\
 & - &  &  0 &  0 \\
 &  &  &  &  0
\end{pmatrix} \]
\normalsize
with the same invariants as $F$. The algorithms in Section~\ref{sec:constr}
for evaluating $\Psi_7$ and $\Psi_{17}$ are implemented in our
Magma function {\tt HesseCovariants(phi,2)}. We use them to compute
$\phi' = 3563 \Psi_7(\phi) + 19 \Psi_{17}(\phi)$.
The algorithms in \cite{minred5} suggest making a 
change of co-ordinates
\small
\[ \begin{pmatrix} x_1 \\ x_2 \\ x_3 \\ x_4 \\ x_5 \end{pmatrix}
\leftarrow
\begin{pmatrix}
-16 & 16 & -9 & -4 & -8 \\
  0 &  0 & -2 & -4 &  0 \\
 24 &  0 &  7 & -4 &  8 \\
  8 &  0 &  2 &  0 &  0 \\
  0 &  0 &  1 &  0 &  0
\end{pmatrix}
\begin{pmatrix} x_1 \\ x_2 \\ x_3 \\ x_4 \\ x_5 \end{pmatrix} \]
\normalsize
so that $\phi'$ becomes
\small
\[ \begin{pmatrix}
0 &  -x_2 + x_3 &  x_1 - x_5 &  x_2 - 2 x_5 &  x_2 - x_4 - x_5 \\
 &  0  & -x_1 + x_2 + x_3 + 2 x_5 &  -x_2 + x_4 + x_5 &  -x_1 - x_2 \\
 &  &  0 &  -x_1 + x_2 + x_4 &  x_3 + x_4 \\
 & - &  &  0 &  x_1 \\
 &  &  &  &  0
\end{pmatrix}. \]
\normalsize

This genus one model has the same invariants as $E$. In particular
its $4 \times 4$ Pfaffians define a curve $C \subset \PP^4$ with 
good reduction at all primes $p \not=2,13,109$. For $p = 2,13,109$ 
we checked directly that $C(\Q_p) \not= \emptyset$.
Since $E(\Q)=0$ it follows that $C$ represents a non-trivial 
element of $\Sha(E/\Q)[5]$. Repeating for 
$P = r_1 P_1 + r_2 P_2$ for $0 \le r_1,r_2 \le 4$ we similarly
find equations for all elements in a subgroup of $\Sha(E/\Q)$ 
isomorphic to $(\Z/5\Z)^2$. 
\end{Example}

The following corollary was already proved in many (but not all)
cases in the appendix to \cite{AS2}.

\begin{Corollary} 
Let $(E,F,n)$ be any of the triples listed in \cite[Table 1]{CM}. 
Then $E$ and $F$ are $n$-congruent, and the visible subgroup of
$H^1(\Q,E)$ explained by $F(\Q)$ is contained in $\Sha(E/\Q)$.
\end{Corollary}
\begin{Proof}
The examples in this table all have $n=3$ or $5$. Using the methods
in \cite{g1hess} and in this paper, we verified the congruences
and computed equations for all relevant elements
of $H^1(\Q,E)$. We then checked directly that these curves 
are everywhere locally soluble.
\end{Proof}

\section{The invariant differential}
\label{sec:invdiff}

The following definition is suggested by the discussion in 
\cite[Section 2]{g1inv}.

\begin{Definition}
Let $C \subset \PP^{n-1}$ be a genus one normal curve with
hyperplane section~$H$. An {\em $\Omega$-matrix} for $C$ is 
an $n \times n$ 
alternating matrix of quadratic forms representing the linear map
\[    \wedge^2 \LL(H) \to \LL(2H)\,; \quad 
f \wedge g \mapsto \frac{f dg - g df}{\omega}  \]
where $\omega$ is an invariant differential on $C$. 
\end{Definition}

Since the natural map $S^2 \LL(H) \to \LL(2H)$ is surjective it
is clear that $\Omega$-matrices exist. However for $n > 3$ 
their entries are only determined up to the addition of quadrics 
vanishing on $C$. Nonetheless we claim in Conjecture~\ref{conj} below
that there is a canonical choice.

From an $\Omega$-matrix we may recover the 
invariant differential $\omega$ using the rule
\[ \omega = \frac{ x_j^2 d(x_i/x_j) }{ \Omega_{ij} } \qquad  \text{ for any } 
i \not= j. \]
We may also characterise $\Omega$-matrices as alternating matrices
of quadratic forms such that
\[ \begin{pmatrix} \partial f/\partial x_0 & \cdots 
&  \partial f/\partial x_{n-1} \end{pmatrix} \Omega = 0\] in $K(C)$ for all 
$f \in I(C)$, and $\Omega$ has rank $2$ at all points on $C$.

Returning to the case $n=5$ this suggests looking for covariants
in the case $Y=\wedge^2 W^*  \otimes S^2 W$.
By Lemma~\ref{heisenchars} we have
$\wedge^2 \theta_4 \otimes S^2 \theta_1 \isom 
2 \theta_3 \otimes 3 \theta_2 \isom  6 \sum_{r,s} \lambda_{r,s}$ 
and so $\dim Y^{H_5}=6$. A basis for $Y^{H_5}$ is
\[ \begin{array}{lll}
\sum (x_1^* \wedge x_4^*) x_0^2, & 
\sum (x_1^* \wedge x_4^*) x_1 x_4, & 
\sum (x_1^* \wedge x_4^*) x_2 x_3,  \\ 
\,\, \sum (x_2^* \wedge x_3^*) x_0^2, & 
\,\, \sum (x_2^* \wedge x_3^*) x_1 x_4, & 
\,\, \sum (x_2^* \wedge x_3^*) x_2 x_3.
\end{array} \]
Since $-I, T \in \Gamma$ act on $Y^{H_5}$ with traces 
$-6$ and $1$ it is easy to see from the character table for 
$\Gamma$ that $Y^{H_5}$ has character $\chi_4 = S^5 \chi_1$. 
The $K[c_4,c_6]$-module of integer weight discrete covariants is 
generated in degrees $5, 15, 15, 25, 25, 35$. Checking
the conditions in \cite{paperII} shows that all of these are covariants. 

The discrete covariant of degree $5$ is
\begin{equation*}
\Omega_5  =
\begin{pmatrix}
0& \alpha_3 & \beta_1 & -\beta_4 & -\alpha_2 \\ 
 -\alpha_3 & 0& \alpha_4 & \beta_2 & -\beta_0 \\ 
 -\beta_1 & -\alpha_4 & 0& \alpha_0 & \beta_3 \\ 
\beta_4 & -\beta_2 & -\alpha_0 & 0 & \alpha_1 \\ 
\alpha_2 & \beta_0 & -\beta_3 & -\alpha_1 & 0 \\
\end{pmatrix} 
\end{equation*}
where 
\begin{equation}
\label{quasiomega}
\begin{array}{rcl} \smallskip
\alpha_i & =&  5 a^4 b w_i^2-10 a^3 b^2 w_{i-1} w_{i+1} +(a^5-3b^5)w_{i-2}w_{i+2}  \\ \beta_i & = &  5 a b^4 w_i^2 -(3a^5+b^5) w_{i-1} w_{i+1} +10 a^2 b^3
w_{i-2} w_{i+2}. \end{array} 
\end{equation}

\begin{Proposition}
\label{prop:om5}
If $\phi \in \wedge^2 V \otimes W$ is non-singular then $\Omega_5(\phi)$
is an $\Omega$-matrix for $C_{\phi} \subset \PP^4$.
\end{Proposition}
\begin{Proof}
By Lemma~\ref{lem:hessorb} and the covariance of $\Omega_5$
it suffices to prove this for $\phi$ a Hesse model.
Let $p_0, \ldots, p_4$ be the quadrics~(\ref{eqn:pi}) defining
$C_\phi$ and $J = (\partial p_i/\partial w_j)$ 
the Jacobian matrix. We checked by direct calculation that all the
entries of $J \Omega_5$ belong to the homogeneous ideal
$I(C_\phi) = (p_0, \ldots, p_4)$. Since $C_\phi \subset \PP^4$ is a 
smooth curve the Jacobian matrix $J$ has rank $3$ at all points of 
$C_\phi$. So $\Omega_5$ has rank at most~2
on $C_\phi$. Since an alternating matrix always has even rank it only remains
to show that $\Omega_5$ is non-zero on $C_\phi$. 
By~(\ref{eqn:pi}) and~(\ref{quasiomega}) it suffices to show that 
\begin{equation}
\label{3by3block}
 \det \begin{pmatrix} 
a b & b^2 & -a^2 \\ 
5 a^4 b & -10 a^3 b^2 & a^5-3b^5  \\
5 a b^4 & -3a^5-b^5 & 10 a^2 b^3 \\
\end{pmatrix} = 18 D  
\end{equation}
is non-zero. Since $\Delta = D^5$ this is clear by Theorem~\ref{thm:inv}
and our assumption that $\phi$ is non-singular.
\end{Proof}

Taking the $4 \times 4$ Pfaffians of $\Omega_5$ and identifying
$\wedge^4 W^* = W$ gives a covariant for $Y = W \otimes S^4 W$. 
This is a scalar multiple of the vector of partial derivatives of 
the secant variety covariant~(\ref{def:S10}). 
This observation not only gives an algorithm for computing $\Omega_5$
(used in Section~\ref{sec:constr}) but also suggested to us
the following conjecture about $\Omega$-matrices for 
genus one normal curves of arbitrary degree.
The $r$th {\em higher secant variety} $\Sec^r C$ of a curve $C$ is the Zariski
closure of the locus of all $(r-1)$-planes spanned by $r$ points
on $C$. Thus the usual secant variety is $\Sec^2 C$.

\renewcommand{\theenumi}{\alph{enumi}}

\begin{Lemma}
\label{lem:sec}
Let $C \subset \PP^{n-1}$ be a genus one normal curve of degree
$n \ge 3$. 
\begin{enumerate}
\item If $n = 2r+1$ is odd then $\Sec^r C \subset \PP^{n-1}$ is a
hypersurface $\{ F = 0 \}$ of degree $n$. 
\item If $n= 2r+2$ is even then $\Sec^r C \subset \PP^{n-1}$ is a 
complete intersection $\{ F_1 = F_2 = 0 \}$ where $F_1$ and $F_2$
each have degree $n/2$.
\end{enumerate}
\end{Lemma}
\begin{Proof}
See \cite{hsec} or \cite{vBH}.
\end{Proof}

Let $R = K[x_0, \ldots, x_{n-1}]$ be the co-ordinate ring 
of $\PP^{n-1}$ and write $R = \oplus_{d \ge 0} R_d$ for its 
usual grading by degree. We require that
morphisms of graded $R$-modules have degree $0$ and write
$R(d)$ for the $R$-module with $e$th graded piece $R_{d+e}$.

\begin{Conjecture}
\label{conj}
Let $C \subset \PP^{n-1}$ be a genus one normal curve of degree
$n \ge 3$, and let $F$, respectively $F_1$ and $F_2$, be as in 
Lemma~\ref{lem:sec}. 
\begin{enumerate}
\item If $n$ is odd then there is a minimal free resolution
\begin{equation*}
 0 \ra R(-2n) \stackrel{P^T}{\ra} R(-n-1)^n \stackrel{\Omega}{\ra}
R(-n+1)^n \stackrel{P}{\ra} R 
\end{equation*}
where $\Omega$ is an alternating matrix of quadratic forms and 
\[P = \begin{pmatrix} 
\partial F/\partial x_0 & \cdots & \partial F/\partial x_{n-1}
\end{pmatrix}. \]
Moreover $P$ is (a scalar multiple of) 
the vector of $(n-1) \times (n-1)$ Pfaffians of 
$\Omega$, and $\Omega$ is an $\Omega$-matrix for $C$. 
\item If $n$ is even then there is a minimal free resolution
\[ 0 \ra R(-n)^2 \stackrel{P^T}{\ra} R(\textstyle\frac{-n+2}{2})^n 
\stackrel{\Omega}{\ra} R(\textstyle\frac{-n+2}{2})^n \stackrel{P}{\ra} R^2 \]
where $\Omega$ is an alternating matrix of quadratic forms and
\[P = \begin{pmatrix} 
\partial F_1/\partial x_0 & \cdots & \partial F_1/\partial x_{n-1} \\
\partial F_2/\partial x_0 & \cdots & \partial F_2/\partial x_{n-1} 
\end{pmatrix}. \]
Moreover the $2 \times 2$ minors of $P$ are (a fixed scalar multiple of)
the $(n-2) \times (n-2)$ Pfaffians of $\Omega$, and $\Omega$ 
is an $\Omega$-matrix for $C$.
\end{enumerate}
\end{Conjecture}

\renewcommand{\theenumi}{\roman{enumi}}

\begin{Remarks}
(i) If $n=3,4$ then the equations in Lemma~\ref{lem:sec}
are the equations for $\Sec^1 C = C$. The conjecture reduces to
some well known formulae for the invariant differential. \\
(ii) If $n = 2r+1$ is odd then it is known (see \cite[Section 8]{vBH}) 
that $\Sec^r C$ has singular locus $\Sec^{r-1} C$ and the latter 
is Gorenstein of codimension $3$. It follows by the 
Buchsbaum-Eisenbud structure
theorem \cite{BE1}, \cite{BE2} that a minimal free resolution of 
the stated form exists. The content of the conjecture is that
the alternating matrix constructed in this way is an $\Omega$-matrix. \\
(iii) If $n=5$ then it suffices to take $C= C_\phi$ with $\phi$
a Hesse model. We have already observed that the $4 \times 4$ Pfaffians
of $\Omega_5(\phi)$ are the partial derivatives of $F=S_{10}(\phi)$.
Combined with (ii) this proves the conjecture in the
case $n=5$. \\
(iv) We have tested the conjecture in some numerical examples
over finite fields for $n = 6,7,8,10,12$.
\end{Remarks}

\section{Explicit constructions}
\label{sec:constr}

We give evaluation algorithms for each
of the covariants in Table~4.6 (as reproduced below). 
This is mainly of interest for the covariants $\Pi_{19}$ 
and $\Pi_{29}$ used in Example~\ref{ex:double} and the covariants 
$\Psi_{7}$ and $\Psi_{17}$ used in Example~\ref{ex:vis}.

\[ \begin{array}{cll}
Y & \text{degrees} & \text{covariants} \\ \hline
\wedge^2 V \otimes W   &  1,11 & U, H\\
V^* \otimes \wedge^2 W   & 7,17 & \Psi_{7}, \Psi_{17} \\
V \otimes \wedge^2 W^*   & 13,23 & \Xi_{13}, \Xi_{23} \\
\wedge^2 V^* \otimes W^* & 19,29 & \Pi_{19}, \Pi_{29} \\ 
V^* \otimes S^2W         & 2,12,22 & P_{2}, P_{12}, P_{22} \\
S^2V \otimes W           & 6,16,26 & Q_{6}, Q_{16}, Q_{26} \\
V \otimes S^2W^*         & 18,28,38 & R_{18}, R_{28}, R_{38} \\
S^2 V^* \otimes W^*      & 14,24,34 & S_{14}, S_{28}, S_{38}
\end{array} \] 

We fix our choice of generators by specifying
them on the Hesse family. The covariant $U$ is the identity map,
whereas the Hessian is given by
\[ H =  
  -\textstyle\frac{\partial D}{\partial b} \sum (v_1 \wedge v_4) w_0 
+ \frac{\partial D}{\partial a} \sum (v_2 \wedge v_3) w_0.  \]
The corresponding formulae for $\Psi_{7}, \Psi_{17}, \Pi_{19}$ and $\Pi_{29}$
are given in~(\ref{def:Psi}) and~(\ref{def:Pi}). We put
\begin{equation*}
\Xi_d = f_d(a,b) \textstyle\sum v_0 (w^*_1 \wedge w^*_4) 
+ g_d(a,b) \sum v_0 (w^*_2 \wedge w^*_3) 
\end{equation*}
where $f_{13}(a,b) = b^3 (26 a^{10} + 39 a^5 b^5 - b^{10})$, 
$g_{13}(a,b) = a^3 (a^{10} + 39 a^5 b^5 - 26 b^{10})$ and
\begin{align*}
f_{23}(a,b) &= -b^3 (46 a^{20} + 1173 a^{15} b^5 - 391 a^{10} b^{10} 
    + 207 a^5 b^{15} + b^{20}),\\
g_{23}(a,b) &= a^3 (a^{20} - 207 a^{15} b^5 - 391 a^{10} b^{10} - 1173 a^5 b^{15} 
  + 46 b^{20}).
\end{align*}

We recall that $P_2$ is the map taking a genus one model to its 
vector of $4 \times 4$ Pfaffians. The remaining generators in
the above table are uniquely determined by the following 
``Pfaffian identities''.
\begin{equation*}
\begin{aligned}
P_2 ( \la U + \mu H) &= \la^2 P_2 + 2 \la \mu P_{12} + \mu^2 P_{22} \\
P_2 ( \la \Psi_7 + \mu \Psi_{17}) 
  & = \la^2 S_{14} + 2 \la \mu S_{24} + \mu^2 S_{34}   \\
P_2 ( \la \Xi_{13} + \mu \Xi_{23}) 
  & = \la^2 Q_{26} - \la \mu (c_6 Q_6 + c_4 Q_{16}) 
                      + \mu^2 (c_4^2 Q_6 + c_6 Q_{16} - c_4 Q_{26}) \\
P_2 ( \la \Pi_{19} + \mu \Pi_{29}) 
  & = \la^2 (c_4 R_{18} + R_{38}) + \la \mu (c_6 R_{18} + c_4 R_{28}) 
   + \mu^2(c_6 R_{28} - c_4 R_{38}) 
\end{aligned}
\end{equation*}

The evaluation algorithms below are justified by checking
them on the Hesse family, and then appealing to Lemma~\ref{lem:hessorb}
and the appropriate covariance properties to show that they 
work for all non-singular models. (We do not consider the case where
the input is singular.)

\begin{Lemma}
\label{lem:indep}
Let $(F,G) = (P_{2},P_{12})$ or $(Q_{6},Q_{16})$. 
If $\phi \in \wedge^2 V \otimes W$ is non-singular and
$F(\phi)$ and $G(\phi)$ are represented by quadratic forms
$f_0, \ldots, f_4$ and $g_0, \ldots, g_4$ in variables $x_0, \ldots, x_4$
then the $75$ quintic forms $\{ (f_i g_j + f_j g_i)x_k : i \le j\}$ 
are linearly independent.
\end{Lemma}

\begin{Proof}
It suffices to check this for $\phi = u(a,b)$ a Hesse model. 

We arrange the coefficients of the quintic forms in a $75 \times 126$ 
matrix. The entries are homogeneous polynomials in $\Q[a,b]$. 
In principle we could finish
the proof by computing the GCD of the $75 \times 75$ minors.
In practice we first decompose the space of quintic forms into
its eigenspaces for the action of $x_i \mapsto \zeta_5^i x_i$. 
This leaves us with one $15 \times 26$ matrix and four $15 \times 25$
matrices.
Rather than compute all $15 \times 15$ minors we compute just those
that correspond to sets of monomials that are invariant
under cyclic permutations of the $x_i$. In each case we find that 
the GCD of these minors divides a power of the discriminant.
\end{Proof}

We gave algorithms for evaluating the invariants $c_4$ and $c_6$ 
in \cite[Section 8]{g1inv} and the Hessian $H$ in \cite[Section 11]{g1hess}.
We compute $P_2, P_{12}$ and $P_{22}$ by taking $4 \times 4$ Pfaffians
of linear combinations of $U$ and $H$ as indicated in the first of the
Pfaffian identities above. We compute $Q_6$ as described 
in \cite[Section 8]{g1inv}. 
Lemma~\ref{lem:indep} shows that we can solve for $Q_{16}$ and 
$Q'_{26}= 5(4 Q_{26} + 3 c_4 Q_6)$ using the identities
\begin{align*}
Q_{16} (P_2,P_{12}) &= Q_6(P_{12},P_{12}), \\
Q'_{26} (P_2,P_{12}) &= Q_6(P_{12}, P_{22}) + 4 Q_{16}(P_{12},P_{12}). 
\end{align*}
Next we use the determinant map 
$V \otimes (V \otimes W) \to S^5 V$ to compute some covariants
taking values in $S^5 V$:  
\begin{align*}
\det ( \la U + \mu Q_6 ) 
   & = \la^4 \mu M_{10} - 2 \la^2 \mu^3 M_{20} + \mu^5 M_{30}, \\
\det ( \la H + \mu Q_6 ) 
   & = \la^4 \mu M_{50} + 2 \la^2 \mu^3 M_{40} + \mu^5 M_{30}, \\
\det ( \la U + \mu Q_{16} ) 
   & = \la^4 \mu M_{20} + 2 \la^2 \mu^3 M'_{50} + \mu^5 M_{80}. 
\end{align*}
Lemma~\ref{lem:indep} shows that we can solve for $R_{18}$ and $R_{28}$ 
using the identities
\begin{align*}
R_{18}(Q_6,Q_{16}) &= \textstyle\frac{\!\!-1\,\,}{18} 
  (5 c_6 M_{10} + 14 c_4 M_{20} + M_{40}), \\
R_{28}(Q_6,Q_{16}) &= 
  \textstyle\frac{\!\!-1\,\,}{792} (9 c_4^2 M_{10} + 620 c_6 M_{20} 
  - 270 c_4 M_{30} + M_{50} - 216 M'_{50}). 
\end{align*}
Then we compute $\Pi_{19}$ and $\Pi_{29}$ using the natural map 
\begin{align*}
( \wedge^2 V \otimes W) \times (V \otimes S^2 W^*) 
& \to \wedge^2 V^* \otimes W^* \\
(U, R_{18}) & \mapsto 2\Pi_{19} \\
(U, R_{28}) & \mapsto 2\Pi_{29}. 
\end{align*}

In Section~\ref{sec:invdiff} we constructed a covariant
$\Omega_5 : \wedge^2 V \otimes W \to  \wedge^2 W^* \otimes S^2 W$.
Conjecture~\ref{conj} (which is a theorem in the case $n=5$)
gives an algorithm for computing $\Omega_5$ (up to sign) using 
minimal free resolutions. 
We may represent $P_{12}$ and $P_{22}$ as $5$-tuples
of quadrics and $\Omega_5$ as a $5 \times 5$ alternating matrix of 
quadrics. Then $\Psi_{7}$ and $\Psi_{17}$ tell us how to write the quadrics in
$P_{12}$ and $P_{22}$ as linear combinations of the quadrics in $\Omega_5$.
In basis-free language there is a natural map
\begin{align*}
( V^* \otimes \wedge^2 W) \times (\wedge^2 W^* \otimes S^2 W) 
& \to V^* \otimes S^2 W \\
(\Psi_7, \Omega_5) & \mapsto P_{12} \\
(\Psi_{17}, \Omega_5) & \mapsto \tfrac{1}{2} (P_{22}+ c_4 P_{2}).
\end{align*}
The proof of Proposition~\ref{prop:om5} shows that if $\phi$ 
is non-singular then the entries of $\Omega_5(\phi)$ 
above the diagonal are linearly independent. We can therefore 
solve for $\Psi_{7}$ and $\Psi_{17}$ by linear algebra. We
then compute $\Xi_{13}$ and $\Xi_{23}$ using the natural map
\begin{align*}
( V^* \otimes \wedge^2 W) \times (S^2 V \otimes W) 
& \to V \otimes \wedge^2 W^* \\
(\Psi_{7}, Q_6) & \mapsto 2 \, \Xi_{13} \\
(\Psi_{7}, Q_{16}) & \mapsto -2 \, \Xi_{23}. 
\end{align*}
Since we only computed $\Omega_5$ up to sign
we have only computed $\Psi_{7}$, $\Psi_{17}$, $\Xi_{13}$ and $\Xi_{23}$ 
up to sign. Fortunately this does not matter for our applications.
(See Example~\ref{ex:vis}.)

The remaining covariants $S_{14}, S_{24}, S_{34}$ and $R_{38}$ 
may be computed using the Pfaffian identities.

\end{document}